\documentclass[reqno,final]{siamltex}
\usepackage[shortlabels]{enumitem}
\usepackage{geometry}
\usepackage{layouts}
\usepackage{stmaryrd}
\usepackage{xcolor}
\usepackage{caption}
\usepackage{graphicx}
\usepackage{diagbox}
\usepackage{cite}
\usepackage{mathrsfs}
\usepackage{amsmath,amssymb,amscd,amsxtra,amsfonts}
\usepackage{tikz}
\usepackage{caption,subcaption}

\providecommand{\Div}{\operatorname{div}}          

\providecommand*{\Dist}[2]{\operatorname{dist}({#1};{#2})}   
\providecommand*{\Dist}[2]{\Dist{#1}{#2}}

\providecommand{\esssup}{\operatorname*{esssup}}

\newcommand{\Bd}{{\boldsymbol{d}}}

\newcommand{\Bn}{{\boldsymbol{n}}}

\newcommand{\Bp}{{\boldsymbol{p}}}

\newcommand{\Bw}{{\boldsymbol{w}}}
\newcommand{\Bx}{{\boldsymbol{x}}}
\newcommand{\By}{{\boldsymbol{y}}}

\newcommand{\BC}{{\boldsymbol{C}}}

\newcommand{\BL}{{\boldsymbol{L}}}

\newcommand{\BW}{{\boldsymbol{W}}}
\newcommand{\BX}{{\boldsymbol{X}}}

\newcommand{\alphabf}{\boldsymbol{\alpha}}
\newcommand{\betabf}{\boldsymbol{\beta}}
\newcommand{\gammabf}{\boldsymbol{\gamma}}

\newcommand{\zetabf}{\boldsymbol{\zeta}}

\newcommand{\chibf}{\boldsymbol{\chi}}

\newcommand{\Ce}{\mathcal{E}}

\newcommand{\Ci}{\mathcal{I}}

\newcommand{\Cl}{\mathcal{L}}

\newcommand{\Cp}{\mathcal{P}}

\newcommand{\Ct}{\mathcal{T}}

\newcommand{\bbG}{\mathbb{G}}

\newcommand{\bbI}{\mathbb{I}}
\newcommand{\bbJ}{\mathbb{J}}

\newcommand{\bbL}{\mathbb{L}}

\newcommand{\bbR}{\mathbb{R}}

\newcommand{\bbW}{\mathbb{W}}

\newcommand*{\N}[1]{\left\|{#1}\right\|}     
\newcommand*{\dN}[1]{\big\|{#1}\big\|}     
\newcommand*{\tN}[1]{\big\|\hspace{-0.09em}\big|{#1}\big|\hspace{-0.09em}\big\|}     
\newcommand*{\TN}[1]{\left\|\hspace{-0.09em}\left|{#1}\right|\hspace{-0.09em}\right\|}     
\newcommand*{\SN}[1]{\left|{#1}\right|}      

\newcommand*{\Lp}[2][\defaultdomain]{L^{#2}({#1})}
\newcommand*{\Lpv}[2][\defaultdomain]{\BL^{#2}({#1})}
\newcommand*{\NLp}[3][\defaultdomain]{\N{#2}_{\Lp[#1]{#3}}}
\newcommand*{\NLpv}[3][\defaultdomain]{\N{#2}_{\Lpv[#1]{#3}}}

\newcommand*{\Ltwo}[1][\defaultdomain]{\Lp[#1]{2}}
\newcommand*{\Ltwov}[1][\defaultdomain]{\Lpv[#1]{2}}
\newcommand*{\NLtwo}[2][\defaultdomain]{\NLp[#1]{#2}{2}}
\newcommand*{\NLtwov}[2][\defaultdomain]{\NLpv[#1]{#2}{2}}

\newcommand*{\Linf}[1][\defaultdomain]{L^{\infty}({#1})}
\newcommand*{\Linfv}[1][\defaultdomain]{\BL^{\infty}({#1})}
\newcommand*{\NLinf}[2][\defaultdomain]{\N{#2}_{\Linf[{#1}]}}
\newcommand*{\NLinfv}[2][\defaultdomain]{\N{#2}_{\Linfv[{#1}]}}

\newcommand*{\Hm}[2][\defaultdomain]{H^{#2}({#1})}

\newcommand*{\Hone}[1][\defaultdomain]{\Hm[#1]{1}}

\newcommand*{\NHone}[2][\defaultdomain]{{\N{#2}}_{\Hone[{#1}]}}

\newcommand*{\SNHone}[2][\defaultdomain]{{\SN{#2}}_{\Hone[{#1}]}}

\newcommand*{\jump}[2][]{\llbracket{#2}\rrbracket_{#1}}

\newcommand{\D}{\mathrm{d}}

\newcommand{\ol}{\overline}

\newcommand{\be}{\begin{eqnarray}}
\newcommand{\ee}{\end{eqnarray}}

\newcommand{\ben}{\begin{eqnarray*}}
\newcommand{\een}{\end{eqnarray*}}

\newcommand*{\markChange}[1]{\textcolor{blue}{#1}}

\newtheorem{assumption}[theorem]{\sc Assumption}
\newtheorem{remark}[theorem]{\sc Remark}
\newtheorem{algorithm}[theorem]{\sc Algorithm}

\def\address#1{\expandafter\def\expandafter\@aabuffer\expandafter
  {\@aabuffer{\affiliationfont{#1}}\relax\par
    \vspace*{13pt}}}

\title{A fourth-order unfitted characteristic finite element method for
solving the advection-diffusion equation on time-varying domains}

\author{Chuwen Ma\thanks{(1) School of Mathematical Science,
    University of Chinese Academy of Sciences.
    (2) Institute of Computational Mathematics and Scientific/Engineering Computing,
    Academy of Mathematics and Systems Science,
    Chinese Academy of Sciences, Beijing, 100190,
    China. The first author was supported by National Key R \& D
    Program of China 2019YFA0709600 and 2019YFA0709602. (chuwenii@lsec.cc.ac.cn)}
  \and Qinghai Zhang\thanks{School of Mathematical Sciences,
    Zhejiang University,
    38 Zheda Road,
    Hangzhou, Zhejiang Province, 310027 China.
    The second author was supported in part by
    China NSF grant 11871429. (qinghai@zju.edu.cn)}
  \and Weiying Zheng\thanks{(1) LSEC, NCMIS,
    Institute of Computational Mathematics and Scientific/Engineering Computing,
    Academy of Mathematics and Systems Science,
    Chinese Academy of Sciences, Beijing, 100190, China.
    (2) School of Mathematical Science,
    University of Chinese Academy of Sciences.
    The third author was supported in part by
    the National Science Fund for Distinguished Young Scholars 11725106,
    by China NSF grant 11831016, and by National Key R \& D
    Program of China 2019YFA0709600 and 2019YFA0709602. (zwy@lsec.cc.ac.cn)}}

\begin{document}
\maketitle

\begin{abstract}
  We propose a fourth-order unfitted characteristic finite element method
  to solve the advection-diffusion equation on time-varying domains.
  Based on a characteristic-Galerkin formulation,
  our method combines the cubic MARS method for interface tracking,
  the fourth-order backward differentiation formula for temporal integration,
  and an unfitted finite element method for spatial discretization.
  Our convergence analysis includes errors of discretely representing
  the moving boundary, tracing boundary markers,
  and the spatial discretization and the temporal integration of the governing equation.
  Numerical experiments are performed on a rotating domain and a severely deformed
  domain to verify our theoretical results
  and to demonstrate the optimal convergence of the proposed method.
\end{abstract}



\section{Introduction}
\label{sec:introduction}

Multiphase flows are ubiquitous in science and engineering
 and the study of them is of great significance
 in a wide range of applications.
One core difficulty to numerical simulation
 is that the domain of each fluid phase may vary in time.
In addition, the potentially large deformations of the domain boundary
 may incur complex interactions of multiple scales both in time and in space.
When the thickness of the interface that separates the fluid from other phases
 is negligible,
 the tracking of the time-varying domain
 can be reduced to that of its boundary;
 in this case the problem is also referred to as a moving boundary problem.

The fidelity of numerically simulating physical processes
 on a time-varying domain
 is very much influenced by the locus of the moving boundary
 within each time step.
On the one hand, interface tracking incurs
 errors that will inevitably affect the accuracy of the entire numerical simulation.
On the other hand, sometimes the velocity field needed for interface tracking
 is not known a priori
 but can only be deduced from the state of the bulk fluid;
 this is especially true for realistic multiphase flows such as
 air-water free-surface flows.
This potentially tight coupling of the fluid and the interface
 in moving boundary problems
 poses great challenges to computational scientists.

There are mainly two approaches
 in solving partial differential equations (PDEs) on domains with
 irregular and moving boundaries.
In the body-fitted methods,
 the discretization mesh is \emph{moved} at each time step
 to follow the time-varying domain
 so that constitutive laws as well as kinematic conditions on the deforming boundary
 can be imposed conveniently.
Another popular category is the unfitted methods,
 in which the underlying mesh, once generated,
 is \emph{fixed} for all time steps
 and the moving interface is allowed to cut cells or elements of
 the static mesh.
Despite its special treatment for cut elements,
 the unfitted methods are attractive in developing high-order schemes.

Unfitted methods have been highly successful.
For stationary interface problems,
 popular unfitted methods include
 the immersed interface method (IIM) \cite{lev94,li06},
 the immersed finite element method (IFEM) \cite{lin09},
 the extended finite element method (XFEM) \cite{dol01,gro11},
 the cut finite element method \cite{bur15},
 the interface-penalty finite element method \cite{han02,wu19},
 the fictitious domain method \cite{has09,bur10} and many others.
For moving boundary problems, Fires and Zilian presented a first-order XFEM method
by using the backward Euler method for time integration \cite{fri09}.
Based on a space-time discontinuous Galerkin discretization,
Lehrenfeld and Reusken \cite{leh13}
proposed a two-dimensional second-order XFEM scheme,
which was extended to three dimensions by Lehrenfeld in \cite{leh15}.
Recently, Guo \cite{guo21} analyzed a backward Euler IFEM
 for solving parabolic moving interface problems.
In \cite{leh19},
 Lehrenfeld and Olshanskii proposed an unfitted finite element method (UFEM)
 via utilizing the backward Euler method for time integration.
More recently, Lou and Lehrenfeld \cite{lou21}
 extended the results in \cite{leh19}
 by combining isoparametric UFEMs with
 $k$th-order backward differentiation formula (BDF-$k$)
 time stepping ($k=1,2,3$);
 a priori error estimates are carried out for $k=2$.
In their work,
 the movements of curved boundaries are analytically prescribed 
 and numerical errors arise only from the PDE discretization.

In spite of their successes,
 unfitted methods are not ready to be deployed
 in the study of realistic multiphase flows yet.
One major roadblock is
 the lack of algorithmic coupling
 of main flow solvers to interface tracking methods.
In current unfitted methods \cite{fre17,leh13,lou21},
 it is usually assumed that explicit, analytic expressions
 have been given \emph{a priori} to fully describe
 the movement of the boundary.
However, this assumption does not hold for all realistic multiphase flows:
 more often than not the movement of the boundary 
 must be determined \emph{on the fly} from state variables of the main
 flow, e.g., the free-surface flows mentioned in the second paragraph.
As the science of multiphase flows evolves towards more and more complex phenomena,
 there is a pressing need for coupling interface tracking algorithms
 to main flow solvers
 so that realistic moving boundary problems
 can be simulated accurately and efficiently.


We answer this need by developing a fourth-order unfitted characteristic finite element method (UCFEM)
 for numerically solving the advection-diffusion equation
 \eqref{cd-model} on time-varying domains.
Our method combines three main components:
 a fifth-order cubic MARS method \cite{zha18} for interface tracking,
 a fourth-order BDF-4 scheme \cite{liu13}
 for integrating a Lagrangian form of the advection-diffusion equation,
 and a UFEM with piecewise bi-quartic functions for spatial
 discretization. 
The computational domain is a fixed Eulerian mesh
 that covers the full movement range of the time-varying domain
 where the advection-diffusion equation holds.

The contributions of this work lie in three aspects.
\begin{enumerate}[leftmargin=8mm]
\item [(a)] Our method is the first fourth-order unfitted method
   for solving moving boundary problems via incorporating
   a fifth-order interface tracking method.
  This is not surprising since fourth- and higher-order interface-tracking algorithms
  have not been available until recently \cite{zha18}.
  To the best of our knowledge,
  these algorithms have been coupled neither to finite element methods
  nor to finite difference/volume methods.
  Although we assume that the movement of the domain boundary
  is prescribed by analytic expressions,
  the interface tracking algorithm adopted in our method
  paves the way to future designs of more sophisticated methods
  that will be able to handle tight couplings of the fluid and the boundary.

\item [(b)] We prove the stability of numerical solutions under the energy norm. In the Lagrangian frame, there is an essential difference between the proofs for second-order and fourth-order schemes. {\it In the latter case, numerical solutions from early time steps must serve as test functions at the present time step, while they do not belong to the present finite element space}. We overcome this difficulty by defining a modified Ritz projection onto the finite element space.

\item [(c)] Our convergence analysis includes
  error estimates not only for boundary representation and tracing boundary markers,
  but also for spatial discretization and temporal integration
  of the governing equation.
  As the main conclusion,
  the overall error is $O(\tau^4)$ under the energy norm for $h=O(\tau)$,
  where $\tau$ and $h$ are the time-step size and the spatial mesh size, respectively.
\end{enumerate}
\vspace{1mm}

The rest of the paper is organized as follows.
In Section \ref{sec:model-problem},
 we formulate the model problem using Lagrangian coordinates.
In Section \ref{sec:interf-track},
 we present the interface tracking algorithm
 and estimate the error between the exact boundary and the numerically
 approximated result.
In Section \ref{sec:fic_FEM},
 we propose the fourth-order UCFEM and
 prove the well-posedness of the discrete problem.
In Section \ref{sec:well_posed}, we define the modified Ritz projection
 and prove the stability of numerical solutions.
Section \ref{sec:err} is devoted to a priori error estimates of numerical solutions.
In Section \ref{sec:num},
 we demonstrate the optimal convergence of our method
 by results of several numerical experiments.

Throughout this paper, $f \lesssim g$ means
$f\le Cg$ with a generic constant $C>0$ independent of
 $\tau$, $h$, and the segment size $\eta$ for interface tracking;
$f\eqsim g$ means that $f \lesssim g$ and $g \lesssim f$ hold simultaneously.
Vector-valued quantities are denoted by boldface symbols,
such as $\Ltwov=\Ltwo^2$,
and matrix-valued quantities are denoted by blackboard bold symbols,
such as $\bbL^2(\Omega)=\Ltwo^{2\times 2}$.

\section{The model problem}
\label{sec:model-problem}

The advection-diffusion equation
with initial and boundary conditions reads
\begin{subequations}\label{cd-model}
  \begin{align}	
    \frac{\partial u}{\partial t} + \Bw \cdot \nabla u- \Delta u  = f
    & \quad \text{in}\;\; \Omega_t, \label{cd-eqn}   \\
    u =0& \quad \text{on}\;\;\Gamma_t, \label{cd-bc}\\
    u(0) =u_0& \quad \text{in}\;\;\Omega_0,  \label{cd-ic}
  \end{align}	
\end{subequations}
where $\Omega_t\subset\bbR^2$ is a bounded and simply-connected domain with time-varying boundary
$\Gamma_t = \partial \Omega_t$,
$\Bw(\Bx,t)$ is a given function satisfying $\Div\Bw=0$,
$u(\Bx,t)$ stands for the tracer transported by the fluid,
and $f(\Bx,t)$ stands for the source term distributed in $\bbR^2$
and has a compact support.
The equation has been scaled
so that the diffusion coefficient before $\Delta u$ is unit.
We restrict the computations to a finite time interval $[0,T]$.

First, we make an assumption on the fluid velocity and the moving boundary. \vspace{2mm}

\begin{center}
\fbox{\parbox{0.975\textwidth}{
\begin{assumption}\label{ass-1}
We assume that $\Bw\in \BC^6(\bbR^2\times [0,T])$ and has compact support, and that
$\Gamma_t$ is $C^4$-smooth for all $t\in [0,T]$.
\end{assumption}}}
\end{center}
\vspace{2mm}

The physical domain is driven by the fluid velocity $\Bw$
and is defined via the flow map
\begin{equation}\label{eq:Omegat}
  \Omega_t :=\left\{\BX(t;0,\Bx): \;\;\Bx\in\Omega_0\right\},
\end{equation}
where $\BX(t;s,\cdot)$ is defined by the solution
to the ordinary differential equations
\begin{equation}\label{eq:X}
\frac{\D}{\D t}\BX(t;s,\Bx)= \Bw\big(\BX(t;s,\Bx),t\big),
    \quad \forall\,t>s\ge 0;\qquad
\BX(s;s,\Bx) = \Bx.
\end{equation}
Since $\Bw$ is $\BC^6$-smooth,
\eqref{eq:X} has a unique solution for every $s\in [0,T]$
and every $\Bx\in\bbR^2$. This implies that $\BX(t;s,\cdot):\Omega_s\to\Omega_t$
is a diffeomorphism.

For any $\Bx_0\in\Omega_0$, we use the flow map $\BX$ to write $\Bx\equiv\Bx(t) = \BX(t;0,\Bx_0)$.
The material derivative of $u$ is defined as
\begin{equation}\label{eq: DuDt}
\frac{\D}{\D t} u(\Bx(t),t) = \frac{\partial u}{\partial t}(\Bx,t)
+ \Bw(\Bx,t) \cdot\nabla_\Bx u(\Bx,t).
\end{equation}
Then \eqref{cd-model} can be written in an equivalent form
\begin{align}\label{cd1-model}	
\frac{\D u}{\D t} - \Delta u  = f
\quad \text{in}\;\; \Omega_t, \qquad
u =0  \quad \text{on}\;\;\Gamma_t, \qquad
u(0) =u_0 \quad \text{in}\;\;\Omega_0.
\end{align}

 \section{The interface tracking algorithm}
\label{sec:interf-track}

In this section, we explain the cubic MARS algorithm that
  we use to approximate the moving boundary.
  Then we establish rigorous error estimate for interface tracking.

Let $t_n=n\tau$, $n=0,1,\cdots,N$, be a uniform partition of the interval $[0,T]$, where $\tau=T/N$ is the time step size. For convenience, we write $\BX^{m,n} := \BX(t_n;t_{m},\cdot)$ for any $n\ge m>0$ and use the shorthand notation $\BX^{n,m} := (\BX^{m,n})^{-1}$.

\subsection{Discrete flow maps}

Let $\BX^{n-1,n}_\tau$ be the approximation of $\BX^{n-1,n}$ such that
$\Bx^n= \BX^{n-1,n}_\tau(\Bx^{n-1})$ is defined by a fifth-order
Runge-Kutta (RK-5) scheme (cf. \cite{ver78}) for solving \eqref{eq:X} from $t_{n-1}$ to
$t_n$. The multi-step discrete flow map is defined as
$\BX^{n-i,n}_\tau = \BX^{n-1,n}_\tau
  \circ \BX^{n-2,n-1}_\tau
  \circ\cdots\circ \BX^{n-i,n-i+1}_\tau$, $1\le i\le n$.
Similarly, the inverse of $\BX^{n-i,n}_\tau$ is denoted by
$\BX^{n,n-i}_\tau := (\BX^{n-i,n}_\tau)^{-1}$.

For any $t\ge s\ge 0$ and $\Bx\in\bbR^2$, the Jacobi matrix of the flow map is defined as
\begin{equation}\label{eq:bbJ}
  \bbJ(t;s,\Bx) := \nabla_\Bx\BX(t;s,\Bx)
  = \bbI +\displaystyle\int_{s}^{t}
  \nabla \Bw(\BX(\xi;s,\Bx),\xi)\, \bbJ(\xi;s,\Bx)\, \D \xi.
\end{equation}
Since $\Div\Bw=0$, we have $\det(\bbJ)\equiv 1$ for all $t\ge s$ (see e.g. \cite{cho90}).
Using Gronwall's inequality and Assumption~\ref{ass-1}, it is easy to show
\begin{align}
  \label{eq:Jnin}
\N{\bbJ(t;s,\cdot)}_{\bbW^{5,\infty}(\bbR^2)}\lesssim 1,\qquad
\N{\bbJ(t_n;t_{n-i},\cdot)-\bbI}_{\bbL^{\infty}(\bbR^2)} \lesssim\tau,\qquad
1\le i\le 4.
\end{align}
Let the Jacobi matrices of $\BX^{n-i,n}$ and
$\BX^{n-i,n}_\tau$ be denoted, respectively, by
\ben
\bbJ^{n-i,n}(\Bx):= \bbJ(t_n;t_{n-i},\Bx), \qquad
\bbJ^{n-i,n}_\tau(\Bx):= \nabla_\Bx\BX^{n-i,n}_\tau(\Bx).
\een
Similarly, the Jacobi matrices of $\BX^{n,n-i}$ and
$\BX^{n,n-i}_\tau$ are denoted, respectively, by
\ben
\bbJ^{n,n-i}:= \left(\bbJ^{n-i,n}\right)^{-1},\qquad
\bbJ^{n,n-i}_\tau:= \left(\bbJ^{n-i,n}_\tau\right)^{-1}.
\een

Since $\BX^{n-1,n}_\tau$ is obtained by the RK-$5$ scheme for
\eqref{eq:X}, the one-step error is $O(\tau^6)$.
For any bounded domain $D\subset\bbR^2$,
standard error estimates give
\begin{align}\label{Xerr}
\NLinfv[D]{\BX^{n-1,n}_\tau-\BX^{n-1,n}} \lesssim \tau^6.
\end{align}
Moreover, taking the gradients of $\BX^{n-1,n},\BX^{n-1,n}_\tau$ with respect to the spatial variable $\Bx$ does not influence the order of temporal error estimates. So we also have
\begin{align}\label{Jerr}
\N{\bbJ^{n-1,n}_\tau-\bbJ^{n-1,n}}_{\bbL^\infty(D)} \lesssim \tau^6.
\end{align}
Combining \eqref{eq:Jnin}--\eqref{Jerr}, we get
\begin{align}\label{eq:Jmn}
\N{\bbJ^{n-i,n}_\tau-\bbI}_{\bbL^{\infty}(D)} \lesssim\tau,\quad
\N{\bbJ^{m,n}_\tau}_{\bbW^{5,\infty}(D)} \lesssim 1,\quad
1\le i\le 4,\quad 0\le m < n.
\end{align}
Their inverses satisfy similar estimates
\begin{align}\label{eq:Jnni}
\N{\bbJ^{n,n-i}-\bbI}_{\bbL^{\infty}(D)} +
      \N{\bbJ^{n,n-i}_\tau-\bbI}_{\bbL^{\infty}(D)}
      \lesssim\tau,\quad
\N{\bbJ^{n,m}}_{\bbL^{\infty}(D)}
      +\N{\bbJ^{n,m}_\tau}_{\bbL^{\infty}(D)}
      \lesssim 1.
\end{align}
Since $\BX^{m,n}_\tau-\BX^{m,n} =\sum_{j=m}^{n-1}(\BX^{j+1,n}_\tau\circ\BX^{j,j+1}_\tau
-\BX^{j+1,n}_\tau\circ\BX^{j,j+1})\circ\BX^{m,j}$, the error estimates for multi-step maps can be obtained similarly
\begin{align}\label{eq:Xerr}
\N{\BX^{m,n}_\tau -\BX^{m,n}}_{\BW^{\mu,\infty}(D)}
+\N{\BX^{n,m}_\tau-\BX^{n,m}}_{\BW^{\mu,\infty}(D)}
\le C(n-m)\tau^6,\qquad \mu=0,1.
\end{align}

\subsection{The cubic MARS algorithm}
\label{sec:spline}

We adopt the cubic MARS algorithm in \cite{zha18}
which constructs a $C^2$-smooth boundary with cubic spline interpolation.
The purpose here is to estimate the error
between the exact boundary and the approximate boundary.

Let $L_0$ be the curve length of $\Gamma_0$. Suppose $\Gamma_0$ has a parametrization
\begin{equation}\label{eq:Gamma0}
\Gamma_0=\left\{\chibf_0(l): l\in [0,L_0]\right\},\qquad \chibf_0\in\BC^4([0,L_0]).
\end{equation}
The interface tracking algorithm starts with a uniform
partition,
$\Cl^0 =\left\{l_j=j\eta:\; j=0,1,\cdots,J_0\right\}$, $\eta=L_0/J_0$,
of the interval $[0,L_0]$
and a set of markers $\Cp^0= \left\{\Bp_j^0:=\chibf_0(l_j): 0\le j\le J_0\right\}$.
\vspace{1mm}

\begin{algorithm}\label{alg:spline}
Given $\Gamma^0_\eta :=\Gamma_0$ and its nodal set
$\Cl^0=\{l^0_j:=l_j: 0\le j\le J_0\}$ and marker set $\Cp^0$, the
cubic MARS algorithm
for constructing $\Gamma^n_\eta$, $n\ge 1$, consists of four steps.
\begin{enumerate}[leftmargin=6mm]
\item Trace forward each marker in $\Cp^{n-1}$
    to obtain the set of markers at $t=t_n$,
    \ben
    \Cp^n =\left\{	\Bp^{n}_{j} =
      \BX^{n-1,n}_\tau(\Bp^{n-1}_{j}):\;
      j=1,\cdots, J_{n-1}\right\},\qquad
      J_n = J_{n-1}.
    \een

\item Adjust the set of markers $\Cp^n$.
\begin{itemize}[leftmargin=4mm]
\item If $M_j:=\left\lceil\SN{\Bp_{j}^n-\Bp^n_{j-1}}/\eta\right\rceil > 1$, create new markers on $\Gamma^{n-1}_\eta$
\ben
\Bp^{n-1}_{j,m} = \chibf_{n-1}\big(l^{n-1}_{j-1}
+m(l^{n-1}_{j}-l^{n-1}_{j-1})/M_j\big),\qquad
1\le m <M_j,
\een
and update $\Cp^n$ as follows
\begin{equation}\label{alg-step2}
 \Cp^n \leftarrow  \Cp^n \cup\left\{
 \BX^{n-1,n}_\tau(\Bp^{n-1}_{j,m}):\;
	1\le m < M_j \right\}, \qquad
    J_n \leftarrow   J_n + M_j -1.
\end{equation}

\item Remove markers from $\Cp^n\backslash\BX^{0,n}_\tau(\Cp^0)$ such that
\ben
0.1\eta < \SN{\Bp^n_{j+1}-\Bp^n_j}\le \eta,\qquad
j =0,\cdots,J_{n}.
\een
\end{itemize}

\item Construct $\Cl^n=\{l^n_j: 0\le j\le J_n\}$ where
$l^n_0=0$ and $l^n_j= l^n_{j-1}+ \SN{\Bp^n_{j}-\Bp^n_{j-1}}$.

\item Compute the cubic spline function $\chibf_n\in\BC^2([0,L_n])$, where $L_n:=l^n_{J_n}$, based on the nodal set $\Cl^n$ and the marker set $\Cp^n$ (see Fig.~\ref{fig:splines}). Construct the approximate boundary by
\begin{equation}\label{eq:ln}
\Gamma^{n}_\eta:=\left\{\chibf_{n}(l): l\in [0,L_n]\right\}.
\end{equation}
\end{enumerate}
\end{algorithm}

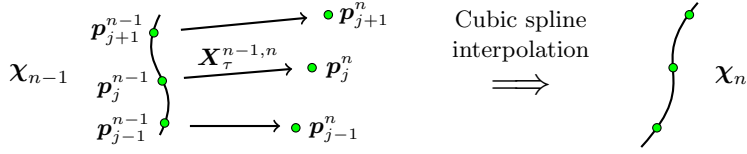
\begin{figure}[http!]
	\centering
	\begin{tikzpicture}[scale =0.8]
	\draw [thick] (-0.5,0) .. controls (0,1) and (-1.0,1) .. (-0.5,2);
	\draw [thick] (7.5,-0.2) .. controls (8.6,1) and (7.5,1.1) .. (8.44,2.2);

	\node at (-2.5,1) {$\chibf_{n-1}$};
	\node at (9,1) {$\chibf_{n}$};

	\draw[fill=green] (-0.42,0.2) circle [radius=0.07];
	\draw[fill=green] (-0.46,0.9) circle [radius=0.07];
	\draw[fill=green] (-0.6,1.7) circle [radius=0.07];
	
	\draw[thick, ->] (0.0,0.15)--(1.5,0.15);
	\draw[thick, ->] (-0.1,0.95)--(1.7,1.1);
	\draw[thick, ->] (-0.15,1.75)--(1.95,1.95);

	\node at (0.8,1.3) {\small $\BX_\tau^{n-1,n}$};

	\node at (-1.1,0.1) {\small $\Bp_{j-1}^{n-1}$};
	\node at (-1.1,0.8) {\small $\Bp_{j}^{n-1}$};
	\node at (-1.2,1.75) {\small $\Bp_{j+1}^{n-1}$};
	
	\node at (2.4, 0.1) {\small $\Bp_{j-1}^{n}$};
	\node at (2.5,1.1) {\small $\Bp_{j}^{n}$};
	\node at (2.9,2.0) {\small $\Bp_{j+1}^{n}$};

	\draw[fill=green] (1.76,0.12) circle [radius=0.07];
	\draw[fill=green] (2.03,1.12) circle [radius=0.07];
	\draw[fill=green] (2.3,2.0) circle [radius=0.07];

	\node at (5.5, 1.9) {\small Cubic spline};
	\node at (5.5, 1.4) {\small interpolation};
    \node at (5.5, 0.8){\Large $\Longrightarrow$};

	\draw[fill=green] (7.76,0.12) circle [radius=0.07];
	\draw[fill=green] (8.03,1.12) circle [radius=0.07];
	\draw[fill=green] (8.3,2.0) circle [radius=0.07];
	\end{tikzpicture}
\caption{An illustration of the interface tracking.}
	\label{fig:splines}
\vspace{-3mm}
\end{figure}

\begin{remark}
The use of cubic spline function in constructing $\Gamma^n_\eta$ has
two benefits. 1) The explicit expression of $\chibf_n$ makes the
computation of integrals on cut elements very efficient. 2)The
$C^2$-smoothness of $\Gamma^n_\eta$ admits a duality
argument in $L^2$-finite element error estimates.
\end{remark}

\subsection{Error estimate for the approximate boundary}

Now we estimate the difference between the approximate boundary $\Gamma^n_\eta$ and the exact boundary $\Gamma_{t_n}$.
{\it The theories of this subsection are restricted to the case without Step~2 in Algorithm~\ref{alg:spline}}.
Here we emphasize that this restriction on Algorithm~\ref{alg:spline} is only required for theoretical analyses, not for numerical computations.
\vspace{2mm}

\begin{center}
\fbox{\parbox{0.975\textwidth}{
\begin{assumption}\label{ass-2}
In theoretical analyses of this paper, we let $\chibf_n\in \BC^2([0,L_0])$ be the cubic spline function computed with
the marker set $\Cp^n =\left\{\Bp^n_j=\BX^{0,n}_\tau(\Bp^0_j): 0\le j\le J_0\right\}$ and the initial nodal set $\Cl^0$. Moreover, the segment size for interface tracking satisfies $\eta=O(\tau^{5/4})$.
\end{assumption}}}
\end{center}
\vspace{2mm}

Assumption~\ref{ass-2} indicates that our theories only apply to the case that $\Gamma_t$ has mild deformations. Remember that the exact boundary is given by
 \begin{align}\label{Xchi}
 \Gamma_{t_n} =\left\{\hat\chibf_n(l):
	0\le l\le L_0\right\},\qquad
	\hat\chibf_n:=\BX^{0,n}\circ\chibf_0.
 \end{align}
Assumption~\ref{ass-1} implies
$\N{\hat\chibf_n}_{\BC^4([0,L_0])}  \lesssim 1$ for all $0\le n\le N$.
By \eqref{eq:Jnin} and \eqref{eq:Jnni}, the arc length of $\Gamma_{t_n}$ between $\BX^{0,n}(\Bp^0_{j-1})$ and $\BX^{0,n}(\Bp^0_{j})$ satisfies
\begin{align*}
\int^{l_j}_{l_{j-1}} \SN{\hat\chibf_n'}	= \int^{l_j}_{l_{j-1}}\SN{\bbJ^{0,n}\chibf_0'}
	\lesssim \eta, \qquad
\eta = \int^{l_j}_{l_{j-1}} \SN{\chibf_0'}
	\lesssim \int^{l_j}_{l_{j-1}}  \SN{\bbJ^{n,0}\hat\chibf_n'}
	\lesssim \int^{l_j}_{l_{j-1}}  \SN{\hat\chibf_n'}.
\end{align*}
This means that $\BX^{0,n}(\Cp_0)$ provides a quasi-uniform partition of $\Gamma_{t_n}$.
Intuitively, the approximation of $\Gamma_{t_n}$ with $\Gamma^n_\eta$
does not deteriorate if we use $\Cp^n=\BX^{0,n}(\Cp_0)$.

Based on the uniform partition $\Cl^0$ of $[0,L_0]$,
$\chibf_n$ can be written explicitly as follows
\begin{align}\label{spline-chi-n}
\chibf_n(l) = \sum_{j=1}^{J_0}\left[\Bp_j^{n}b_j +
\alphabf_j^{n}b_j(b_j^2-1)\right], \qquad \alphabf_0=\alphabf_{J_0},
\end{align}
where $b_j\in C([0,L_0])$ is linear on each sub-interval $[l_{i-1},l_i]$ and satisfies
$b_j(l_i) =\delta_{i,j}$.
Define ${\underline\Bd}^{n}
=\left[\Bd_1^{n},\cdots,\Bd_{{J_0}}^{n}\right]^\top$ with
$\Bd_j^{n}= \Bp_{j+1}^n+\Bp_{j-1}^n-2\Bp_{j}^n$ and $\Bp_{{J_0}+1}^n=\Bp_1^n$.
The coefficient tensor $\underline\alphabf^{n} =\left[
\alphabf_1^{n},\cdots,\alphabf_{{J_0}}^{n}\right]^\top$ solves the system of algebraic equations
\begin{equation}	\label{eq:Gad}
	\bbG {\underline\alphabf}^{n} = {\underline\Bd}^{n} ,
\end{equation}
where $\bbG$ is the $(2{J_0})\times(2{J_0})$ matrix
\begin{equation}\label{eq:G}
\bbG=
\begin{bmatrix}
		4     &0     &1   &       &       &1   & 0\\
		0     &4     &0     &1    &       &   &1\\
		1   &0     &4     &0      &1    &      & \\
		&\ddots&\ddots&\ddots &\ddots &\ddots& \\
		&      &1   &0      &4      &0     &1 \\
		1   &      &      &1    &0      &4     &0  \\
		0 &1   &      &       &1    &0     &4
\end{bmatrix} .
\end{equation}
To measure the difference between $\Gamma^n_\eta$ and $\Gamma_{t_n}$, it suffices to estimate the error $\chibf_n-\hat\chibf_n$. \vspace{1mm}

\begin{lemma}\label{lem:bdr}
Let Assumptions~\ref{ass-1} and \ref{ass-2} be satisfied. Then
  \begin{align}\label{eq:Gamma-tn-err}
    \N{\chibf_n-\hat\chibf_n}_{\BC^{\mu}([0,L_0])}
    \lesssim \eta^{4-\mu}+\tau^5,\quad 0\le n\le N,
    \quad \mu=0,1, 2.
  \end{align}
\end{lemma}
\begin{proof}
Define $\tilde\chibf_n(l):= \BX^{0,n}_\tau(\chibf_0(l))$. Clearly $\chibf_n$ is the cubic spline interpolation of $\tilde\chibf_n$. By the chain rule, \eqref{eq:Jmn}, and standard error estimates for cubic spline interpolations, we have
\begin{equation}\label{eq:tchi-err}
\N{\tilde\chibf_n}_{\BC^4([0,L_0])} \lesssim 1,\qquad
\N{\chibf_n-\tilde\chibf_n}_{\BC^{\mu}([0,L_0])}\lesssim \eta^{4-\mu}.
\end{equation}
The proof is finished by using the triangular inequality and \eqref{eq:Xerr}.
\end{proof}

\begin{theorem}\label{thm:chi-err}
Let $\eta=O(\tau^{5/4})$.
The composite function $\BX_{\tau}^{m,n}\circ\chibf_{m}(l)=\BX_{\tau}^{m,n}(\chibf_{m}(l))$ satisfies
\begin{equation}\label{chi-err}
\N{\chibf_n - \BX_{\tau}^{m,n}\circ\chibf_m}_{\BC^{\mu}([0,L_0])}
    \lesssim \tau^{6-5\mu/4},\quad  0\le n-m\le 4, \quad \mu=0,1.
  \end{equation}
\end{theorem}
\begin{proof}
First we estimate $\chibf_n-\chibf_m$.
For fixed $s\ge 0$ and $\Bp$, we define a univariate function of $t$ by
$\BW(t;s,\Bp):= \Bw(\BX(t;s,\Bp), t)$. Let $\BW^{(i)}(t;s,\Bp)$ denote the $i^{\rm th}$-order derivative of $\BW(t;s,\Bp)$ with respect to $t$.
Assumption~\ref{ass-1} and the chain rule show that
  \ben
  \big|\BW^{(1)}(t;s,\Bp)\big| = \Big|(\Bw\cdot\nabla\Bw)(\BX(t;s,\Bp),t)
    +\frac{\partial\Bw}{\partial t}(\BX(t;s,\Bp),t)\Big|\lesssim 1.
  \een
  High-order derivatives of $\BW(\cdot;s,\Bp)$
  can be estimated similarly. Using \eqref{eq:X}, we have
  \ben
  \N{\BX(\cdot;s,\Bp)}_{\BC^7([s,T])}
  +\N{\BW(\cdot;s,\Bp)}_{\BC^6([s,T])}
  \lesssim 1 .
  \een

By \eqref{eq:Xerr} and Taylor's expansion of $\BX(t_n;t_m,\Bp)$ at $t_m$, we have
\begin{equation}\label{eq:Xp}
\BX_\tau^{m,n}(\Bp) = \BX(t_n;t_m,\Bp) +O(\tau^6)
= \Bp + \sum_{i=0}^4\frac{\BW^{(i)}(t_m;t_m,\Bp)}{(i+1)!}(t_n-t_m)^{i+1} +O(\tau^6).
\end{equation}
Since $\N{\bbG^{-1}}_{\infty}\lesssim 1$, taking $\Bp=\Bp^n_j\equiv \BX^{m,n}_\tau(\Bp^{m}_j)$ in \eqref{eq:Xp} and using the definitions of $\Bd_j^n$ and $\Bd_j^{m}$, we find that
\begin{align}\label{eq:alpha2}
\underline\alphabf^{n} - \underline\alphabf^{m}
    = \bbG^{-1}(\underline\Bd^n - \underline\Bd^{m})
    = \sum_{i=0}^{4}\frac{(t_n-t_m)^{i+1}}{(i+1)!}\underline{\gammabf_i}
+O(\tau^6),
\end{align}
where $\underline{\gammabf_i}=\bbG^{-1}\underline{\betabf_i}$ and $\underline{\betabf_i}=
\left[\betabf_{i,1},\cdots,\betabf_{i,{J_0}}\right]^\top$. Each component of $\underline{\betabf_i}$ is defined as
\ben
\betabf_{i,j} = \BW^{(i)}(t_{m};t_{m},\Bp_{j+1}^{m})
+\BW^{(i)}(t_{m};t_{m},\Bp_{j-1}^{m})
-2\BW^{(i)}(t_{m};t_{m},\Bp_{j}^{m}),
\een
where $\Bp_{{J_0}+1}^{m} :=\Bp_{1}^{m}$.
Using the equality $|b_j'|=\eta^{-1}$ and \eqref{spline-chi-n}, we immediately get
\begin{align}\label{chi-chi-0}
\chibf_n^{(\mu)} -\chibf_{m}^{(\mu)}
 = \sum_{i=0}^{4}\frac{(t_n-t_m)^{i+1}}{(i+1)!}
    \zetabf_i^{(\mu)} +O(\eta^{-\mu}\tau^6), \qquad \mu=0,1,
\end{align}
where $\zetabf_i= \sum_{j=1}^{J_0}\big[\BW^{(i)}(t_{m};t_{m},\Bp_j^{m})b_j
    + \gammabf_{i,j}b_j(b_j^2-1)\big]$ is a cubic spline function on $[0,L_0]$.

Note from \eqref{eq:tchi-err} that $\tilde\chibf_m=\BX^{0,m}_\tau\circ\chibf_0
\in\BC^4([0,L_0])$ and $\Bw\in\BC^6(\bbR^2\times[0,T])$. For $0\le i\le 4$,
$\BW^{(i)}(t_m;t_m,\tilde\chibf_m)$ defines a function in
$\BC^{\min(4,6-i)}([0,L_0])$. Moreover, since
\ben
\tilde\chibf_{m}(l_j)= \BX^{0,m}_\tau(\chibf_0(l_j))
=\BX^{0,m}_\tau(\Bp^0_j)=\Bp_j^{m},\qquad 0\le j\le {J_0},
\een
$\zetabf_i$ is actually the cubic spline interpolation of
$\BW^{(i)}(t_m;t_m,\tilde\chibf_m)$.
Standard error estimates yield
  \begin{align}
    \dN{\zetabf_i-\BW^{(i)}(t_{m};t_{m},
    \tilde\chibf_{m})}_{\BC^\mu([0,L_0])}
    \lesssim  \eta^{-\mu}\eta^{\min(4,6-i)}
    \lesssim \eta^{-\mu}\big(\tau^5+ \tau^{7.5-1.25i}\big),
    \label{zeta-m-err}
  \end{align}
where we have used the relation $\eta=O(\tau^{5/4})$. Combining \eqref{chi-chi-0} and \eqref{zeta-m-err} yields
  \begin{equation}
    \chibf_n^{(\mu)}- \chibf_{m}^{(\mu)} =
    \sum_{i=0}^4\frac{(t_n-t_m)^{i+1}}{(i+1)!}
    \frac{\D^\mu}{\D l^\mu}\BW^{(i)}(t_{m};t_{m},\tilde\chibf_{m})
    +O\big(\eta^{-\mu}\tau^6\big),\quad \mu=0,1.
    \label{chi-chi-2}
  \end{equation}

Applying \eqref{eq:Xp} to $\BX_\tau^{m,n}(\chibf_{m})$ and using
\eqref{chi-chi-2} and \eqref{eq:tchi-err}, we find that
\begin{equation*}
|\chibf_n-\BX_\tau^{m,n}\circ\chibf_{m}|
    \lesssim \sum_{i=0}^4\tau^{i+1}
    |\BW^{(i)}(t_{m};t_{m},\tilde\chibf_{m})
    -\BW^{(i)}(t_{m};t_{m},\chibf_{m})|  +\tau^6
    \lesssim \tau^6.
  \end{equation*}
Since $\Bw\in\BC^6(\bbR^2\times [0,T])$, similar to \eqref{eq:Xp},
  we also have
  \begin{align*}
    \nabla_\Bp\BX_\tau^{m,n}(\Bp) = \bbI + \sum_{i=0}^4
    \frac{(t_n-t_m)^{i+1}}{(i+1)!}\nabla_\Bp\BW^{(i)}(t_{m};t_{m},\Bp)
    +O(\tau^6) .
  \end{align*}
Using \eqref{chi-chi-2} and \eqref{eq:tchi-err},
the derivative of $\chibf_n-\BX_\tau^{m,n}\circ\chibf_{m}$ can be estimated similarly
  \begin{equation*}
    \begin{aligned}
      \SN{\chibf_n' -(\BX_\tau^{m,n}\circ\chibf_{m})'}
      \lesssim \,&\sum_{i=0}^4\tau^{i+1}
      \Big|\frac{\D}{\D l}\BW^{(i)}(t_m;t_m,\tilde\chibf_m)
        -\frac{\D}{\D l}\BW^{(i)}(t_m;t_m,\chibf_m)\Big|
      + \eta^{-1}\tau^6\\
      \lesssim \,&\sum_{i=0}^4\tau^{i+1}
      \SN{\tilde\chibf_m' -\chibf_m'}+ \eta^{-1}\tau^6
      \lesssim \eta^{-1}\tau^6.
    \end{aligned}
  \end{equation*}
  The proof is finished by using $\eta=O(\tau^{5/4})$.
\end{proof}

\section{The unfitted characteristic finite element method}
\label{sec:fic_FEM}

The purpose of this section is to propose the UCFEM
for solving \eqref{cd-model} on a fixed mesh.
First we take an open square $D\subset\bbR^2$ which is large enough
such that $\Omega_t\cup\Omega^n_\eta \subset D$ for all $0\le t\le T$ and $0\le n\le N$.

\subsection{Finite element spaces}

Let $\Ct_h$ be the uniform partition of $\bar D$ into
{\em closed} squares of side-length $h$.
Let $\tilde\Omega^n_\eta\subset D$ be a domain slightly larger than $\Omega^n_\eta$:
\begin{equation}\label{tdomain}
\tilde\Omega^n_\eta :=\left\{\Bx\in\bbR^2:
    \mathrm{dist}\big(\Bx,\ol{\Omega^n_\eta}\big)< h/2\right\},\qquad
\tilde\Gamma^n_\eta :=\partial\tilde\Omega^n_\eta
\end{equation}
Clearly $\Ct_h$ generates a cover of $\tilde\Omega^n_{\eta}$
and a cover of $\Gamma^n_{\eta}\cup\tilde\Gamma^n_\eta$, respectively,
(see Fig.~\ref{fig:mesh})
\ben
&&\Ct^n_h := \left\{K\in\Ct_h:\;
\mathrm{area}(K\cap\tilde\Omega^n_{\eta})>0\right\}, \\
&&\Ct^n_{h,B} := \left\{K\in \Ct^n_h:\;
\mathrm{length}(K\cap\Gamma^n_{\eta}) >0\;\;\hbox{or}\;\;
\mathrm{length}(K\cap\tilde\Gamma^n_{\eta}) >0\right\} .
\een
Then $\Ct^n_h$ generates a polygonal domain $\Omega^n_h$ which contains $\tilde\Omega^n_\eta$:
\ben
\Omega^n_h:= \mathrm{interior}
\big(\cup_{K\in\Ct^n_h}K\big),\qquad
\Gamma^n_h := \partial\Omega^n_h .
\een
Let $\Ce_h$ be the set of all edges in $\Ct_h$ and define
\ben
\Ce_{h,B}^{n}= \big\{E\in\Ce_h: \; E\not\subset \Gamma^n_h\;\;\hbox{and}\;\;
\exists K\in \Ct^n_{h,B}\;\;
\hbox{s.t.}\;\; E\subset\partial K\big\}.
\een

\begin{figure}[http!]\label{fig:sub12}
	\centering
\begin{subfigure}{.3\textwidth}
		\centering
\begin{tikzpicture}[scale =1.7]
\filldraw[red!30!](0.2*2,0.2*2)--(0.2*8,0.2*2)--(0.2*8,0.2*3)--(0.2*9,0.2*3)--(0.2*9,0.2*7)
    --(0.2*8,0.2*7)--(0.2*8,0.2*8)--(0.2*2,0.2*8)--(0.2*2,0.2*7)--(0.2*1,0.2*7)--(0.2*1,0.2*3)
    --(0.2*2,0.2*3)--(0.2*2,0.2*2);

\filldraw[white](0.2*3,0.2*6)--(0.2*4,0.2*6)--(0.2*4,0.2*7)--(0.2*5,0.2*7)--(0.2*6,0.2*7)--(0.2*6,0.2*6)
    --(0.2*7,0.2*6)--(0.2*7,0.2*5)--(0.2*7,0.2*4)--(0.2*6,0.2*4)--(0.2*6,0.2*3)--(0.2*5,0.2*3)--(0.2*4,0.2*3)
    --(0.2*4,0.2*4)--(0.2*3,0.2*4)--(0.2*3,0.2*5)--(0.2*3,0.2*6);

\draw[black, thick] (1,1) ellipse [x radius=0.56cm, y radius=0.5cm];
\draw [step =0.2cm,gray,thin] (0,0) grid (2cm,2cm);
			
\draw[red,thick](0.2*2,0.2*2)--(0.2*8,0.2*2)--(0.2*8,0.2*3)--(0.2*9,0.2*3)--(0.2*9,0.2*7)
    --(0.2*8,0.2*7)--(0.2*8,0.2*8)--(0.2*2,0.2*8)--(0.2*2,0.2*7)--(0.2*1,0.2*7)--(0.2*1,0.2*3)
    --(0.2*2,0.2*3)--(0.2*2,0.2*2);

			\node[left] at (0.3,1.7) {$D$};
			\node[left] at (1.2,0.2*5) {$\Omega_\eta^n$};
			\node[right] at (1.2,0.2*7.35){$\Gamma_\eta^n$};
		\end{tikzpicture}
\end{subfigure}
\qquad
\begin{subfigure}{.3\textwidth}
\centering
\begin{tikzpicture}[scale =2.5]
\filldraw[red!30!](0.2*2,0.2*2)--(0.2*8,0.2*2)--(0.2*8,0.2*3)--(0.2*9,0.2*3)--(0.2*9,0.2*7)
    --(0.2*8,0.2*7)--(0.2*8,0.2*8)--(0.2*2,0.2*8)--(0.2*2,0.2*7)--(0.2*1,0.2*7)--(0.2*1,0.2*3)
    --(0.2*2,0.2*3)--(0.2*2,0.2*2);

\filldraw[yellow](0.2*3,0.2*6)--(0.2*4,0.2*6)--(0.2*4,0.2*7)--(0.2*5,0.2*7)--(0.2*6,0.2*7)--(0.2*6,0.2*6)
    --(0.2*7,0.2*6)--(0.2*7,0.2*5)--(0.2*7,0.2*4)--(0.2*6,0.2*4)--(0.2*6,0.2*3)--(0.2*5,0.2*3)--(0.2*4,0.2*3)
    --(0.2*4,0.2*4)--(0.2*3,0.2*4)--(0.2*3,0.2*5)--(0.2*3,0.2*6);
			
\draw[red,thick](0.2*2,0.2*2)--(0.2*8,0.2*2)--(0.2*8,0.2*3)--(0.2*9,0.2*3)--(0.2*9,0.2*7)
    --(0.2*8,0.2*7)--(0.2*8,0.2*8)--(0.2*2,0.2*8)--(0.2*2,0.2*7)--(0.2*1,0.2*7)--(0.2*1,0.2*3)
    --(0.2*2,0.2*3)--(0.2*2,0.2*2);

\draw[black, thick] (1,1) ellipse [x radius=0.56cm, y radius=0.5cm];
\draw[step =0.2cm,gray,thin] (0.6,0.6) grid (7*0.2cm,7*0.2cm);
\draw[blue,thick](0.2*2,0.2*3)--(0.2*2,0.2*7);
\draw[blue,thick](0.2*3,0.2*2)--(0.2*3,0.2*8);
\draw[blue,thick](0.2*4,0.2*2)--(0.2*4,0.2*4); \draw[blue,thick](0.2*4,0.2*6)--(0.2*4,0.2*8);
\draw[blue,thick](0.2*5,0.2*2)--(0.2*5,0.2*3); \draw[blue,thick](0.2*5,0.2*7)--(0.2*5,0.2*8);
\draw[blue,thick](0.2*6,0.2*2)--(0.2*6,0.2*4); \draw[blue,thick](0.2*6,0.2*6)--(0.2*6,0.2*8);
\draw[blue,thick](0.2*7,0.2*2)--(0.2*7,0.2*8);
\draw[blue,thick](0.2*8,0.2*3)--(0.2*8,0.2*7);

\draw[blue,thick](0.2*2,0.2*3)--(0.2*8,0.2*3);
\draw[blue,thick](0.2*1,0.2*4)--(0.2*4,0.2*4); \draw[blue,thick](0.2*6,0.2*4)--(0.2*9,0.2*4);
\draw[blue,thick](0.2*1,0.2*5)--(0.2*3,0.2*5); \draw[blue,thick](0.2*7,0.2*5)--(0.2*9,0.2*5);
\draw[blue,thick](0.2*1,0.2*6)--(0.2*4,0.2*6); \draw[blue,thick](0.2*6,0.2*6)--(0.2*9,0.2*6);
\draw[blue,thick](0.2*2,0.2*7)--(0.2*8,0.2*7);

			\node[left] at (1.15,1.0) {$\Omega_h^n$};
			\node[left] at (2.0,0.47) {\small $\Ct_h^n$};
			\node[right] at (1.18,1.5) {\small $\Ct_{h,B}^n$};
		\end{tikzpicture}
	\end{subfigure}
\caption{Left: the square domain $D$ and its partition $\Ct_h$, the approximate boundary $\Gamma^n_\eta$, and the approximate domain $\Omega_\eta^n$ surrounded by $\Gamma^n_\eta$.
Right: the set of red and yellow squares $\Ct^n_h$, the set of red squares $\Ct^n_{h,B}$, the set of blue edges $\Ce_{h,B}^{n}$, and $\ol{\Omega_h^n}=$ the union of red and yellow squares.} \label{fig:mesh}\vspace{-7mm}
\end{figure}
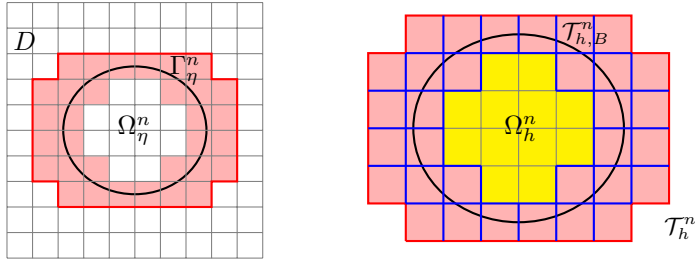

The finite element spaces on $D$ and on $\Omega^n_h$ are,
respectively, defined as
\begin{align*}
{V_h}:= \big\{v\in\Hone[D]:
	v|_K\in Q_4(K),\;\forall\,K\in \Ct_h\big\}, \qquad
{V^n_h} := \big\{v|_{\Omega^n_h}: v\in V_h\big\} ,
\end{align*}
where $Q_4$ is the space of polynomials
whose degrees are no more than $4$ for each variable.
The space of piecewise regular functions over $\Ct^n_h$ is defined as
\ben
H^m(\Ct^n_h) := \big\{v\in\Ltwo[\Omega^n_h]:\;
v|_K\in H^m(K),\;\forall\, K\in\Ct^n_h\big\},\qquad m\ge 1.
\een
Throughout the paper,
we extend $v_h\in {V^n_h}$ to the exterior of $\Omega^n_h$
such that the extension, denoted still by $v_h$,
belongs to ${V_h}$ and vanishes at degrees of freedom outside of
$\ol{\Omega^n_h}$. It is easy to see that
\begin{align}	\label{extension}
	\NLtwo[D]{v_h}\lesssim \NLtwo[\Omega^n_h]{v_h}, \qquad
	\NHone[D]{v_h}\lesssim \NHone[\Omega^n_h]{v_h}.
\end{align}

\subsection{The discrete problem}

We define four bilinear forms on
$H^5(\Ct^n_h) \cap\Hone[\Omega^n_h]$ as follows
\begin{align}	
\mathscr{A}^n_h(w,v):=\,&
(\nabla w,\nabla v)_{\Omega^n_{\eta}} + \mathscr{S}^{n}_h(w,v)
	+ \mathscr{J}_{0}^{n}(w,v)	+ \mathscr{J}_{1}^{n}(w,v) ,\label{A-nh}\\
	\mathscr{S}^n_h(w,v):=\,&- \int_{\Gamma_\eta^n}
	\left(v\partial_\Bn w  + w\partial_\Bn v \right),
	\label{S-nh}\\
	\mathscr{J}^n_0(w,v):=\,&\frac{\gamma_0}{h}
	\int_{\Gamma^{n}_\eta} w v ,
	\label{J0-nh}\\
	\mathscr{J}^n_1(w,v):=\,& \gamma_1
	\sum_{E\in \Ce_{h,B}^{n}}
	\sum_{l=1}^4 h^{2l-1}\int_E
	\jump{\partial_{\Bn}^l w}
	\jump{\partial_{\Bn}^l v}, \label{J1-nh}
\end{align}	
where $\gamma_0,\gamma_1$ are positive constants and
$\partial_{\Bn} v$ denotes the normal derivative of $v$ on $\Gamma^n_\eta$.
In \eqref{J1-nh}, $\partial_{\Bn}^l v$ denotes
the $l$-th order normal derivative of $v$ on $E$ and $\jump{\partial_{\Bn}^l v}$ denotes the jump of $\partial_{\Bn}^l v$ across $E$.
Here $\mathscr{J}^n_0$ is used to
impose the Dirichlet boundary condition of $u^n_h$ weakly,
$\mathscr{J}^n_1$ is used to
enhance the stability of $u^n_h$ (see section~\ref{sec:well_posed}).
Moreover, $(\cdot,\cdot)_{\Omega^n_{\eta}}$ stands for the inner product on $\Ltwo[\Omega^n_{\eta}]$.
\vspace{1mm}

\begin{center}
\fbox{\parbox{0.975\textwidth}{\quad
The UCFEM for \eqref{cd-model} is to
seek  $u^n_h\in V_h^{n}$ such that
\begin{align}	\label{weak-nh}
\frac{1}{\tau}\sum_{i=0}^4 \lambda_i \big(U_h^{n-i,n}, v_h\big)_{\Omega^n_{\eta}}
	+ \mathscr{A}^n_h(u_h^n,v_h)
	= (f^n,v_h)_{\Omega^n_{\eta}} \qquad
	\forall\,v_h\in V_h^{n},
\end{align}	
where $(\lambda_0,\lambda_1,\lambda_2,\lambda_3,\lambda_4)=(25/12,\,-4,\, 3,\, -4/3,\, 1/4)$, $f^n=f(t_n)$, and
\ben
U_h^{n-i,n}(\Bx):=\big(u_h^{n-i}\circ\BX_\tau^{n,n-i}\big)(\Bx)
 =u_h^{n-i}(\BX_\tau^{n,n-i}(\Bx)). \qquad \hbox{(see Fig.~\ref{fig:Uh})}
\een}}
\end{center}

\begin{figure}[http!]
	\centering
	\begin{tikzpicture}[scale =1.3]
		\draw [thick] (0,0) to [out=45,in=190] (0.8,1);
		\draw[thick,->] (0.3,0.5)--(0.25,0.4);
		\draw[fill=green] (0,0) circle [radius=0.05];
		\draw[fill=red] (0.8,1) circle [radius=0.05];

	\node at (1.4,1) {\small $\Bx\in\Omega^n_\eta$};

	\draw[fill =white] (2,0.65) rectangle (6.8,1.15);
	\node at (4.4,0.9) {\small $U_h^{n-i,n}(\Bx) :=u_h^{n-i}\circ\BX^{n,n-i}_h(\Bx) = u_h^{n-i}(\By)$};

	\node at (1,0) {\small $\By=\BX_\tau^{n,n-i}(\Bx)$};	

	\draw[fill =white] (-1.3, -0.05) rectangle (-0.25,0.45);
	\node at (-0.78,0.2) {\small $u_h^{n-i}(\By)$};
	\end{tikzpicture}
\caption{An illustration of $U_h^{n-i,n}(\Bx)$:
calculate $\By=\BX_\tau^{n,n-i}(\Bx)$ and evaluate $u^{n-i}_h(\By)$.}
	\label{fig:Uh}
\vspace{-7mm}
\end{figure}

\subsection{Some discussions of quadratures}

Numerical solutions of PDEs on moving domains are usually time-consuming due to quadratures on ``irregular domains'' or ``cut elements''. The issue becomes even more extrusive for high-order methods than low-order methods. Now we explain that the quadratures in our method can be done efficiently.

The most time-consuming computations involve the first term on the left-hand side of \eqref{weak-nh}:
\begin{align*}
\int_{\Omega^n_\eta}U^{n-i,n}_h v_h
=\sum_{K\in\Ct^n_h} \int_{K\cap\Omega^n_\eta}U^{n-i,n}_h v_h,
\quad 0\le i\le 4,\quad v_h\in {V^n_h}.
\end{align*}
We consider the quadratures on interior elements and
boundary elements, respectively.
\begin{enumerate}[leftmargin=6mm]
\item {\it Quadratures on interior elements}
The integrand $U_h^{n-i,n}v_h$ is continuous and piecewise smooth on any interior element $K\subset\Omega^n_\eta$. We compute $\int_K U_h^{n-i,n}v_h$ directly with the $(2k+3)^{\text{th}}$-order Gauss-Legendre quadrature on $K$.
\vspace{1mm}

\item {\it Quadratures on cut elements.} Suppose $K\in \Ct^n_{h,B}$.
Since $\Gamma^n_\eta$ is constructed with Algorithm~\ref{alg:spline}, the curved edges $K\cap\Gamma^n_\eta$ are represented explicitly by the piecewise cubic function $\chibf_n$. Therefore, $K\cap\Omega^n_\eta$ can be easily subdivided into the union of several X-type and Y-type sub-regions (see Fig.~\ref{fig:cutElement}). We compute the integral on each sub-region with the $(2k+3)^{\text{th}}$-order Gauss-Legendre quadrature.
\end{enumerate}

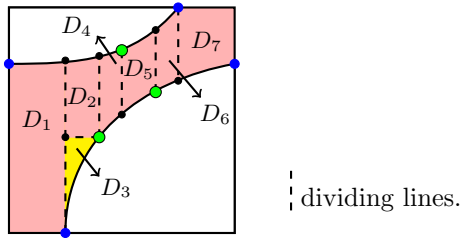
\begin{figure}[http!]
	\centering
	\begin{tikzpicture}[scale =1.5]
		\draw[thick,fill = red!30!white] (0,0) rectangle (2,2);
		\filldraw[yellow](0.5,0) rectangle (0.8,2);
		\filldraw[red!30!white] (0.5,0.85) rectangle (0.8,2);
		\draw [thick,fill=white] (0.5,0) to [out=90,in=190](2,1.5) --(2,0)--(0.5,0);
		\draw [thick,fill=white] (0,1.5) to [out=0,in=230](1.5,2)--(0,2)--(0,1.5);
        \draw[thick,dashed] (0.5,0)--(0.5,1.54);
        \draw[thick,dashed] (0.8,0.85)--(0.8,1.58);
        \draw[thick,dashed] (1.3,1.25)--(1.3,1.8);	
        \draw[thick,dashed] (1,1.62)--(1,1.1);
        \draw[thick,dashed] (1.5,2)--(1.5,1.4);
        \draw[thick,dashed] (0.8,0.85)--(0.5,0.85);
        \draw [blue,fill=blue] (0.5,0) circle [radius=0.04];
        \draw [blue,fill=blue] (2,1.5) circle [radius=0.04];	
        \draw [blue,fill=blue] (0,1.5) circle [radius=0.04];
        \draw [blue,fill=blue] (1.5,2) circle [radius=0.04];

        \draw[black,fill=black](0.5,1.53) circle[radius=0.03];
        \draw[black,fill=black](0.8,1.57)circle[radius=0.03];
        \draw[black,fill=black](1.3,1.8)circle[radius=0.03];
        \draw[black,fill=black](1,1.05)circle[radius=0.03];

        \draw[black,fill=black](0.5,0.85)circle[radius=0.03];

        \draw[black,fill=green] (1,1.62) circle[radius=0.05];
        \draw[black,fill=green] (0.8,0.85) circle [radius=0.05];
        \draw[black,fill=green] (1.3,1.25) circle [radius=0.05];

        \draw[black,fill=black] (1.5,1.35) circle [radius=0.03];

         \node at (0.25,1) {\small $D_1$};
         \node at (0.65,1.2) {\small $D_2$};
         \draw[->,thick] (0.6,0.75)--(0.8,0.5);
         \node at (0.95,0.4) {\small $D_3$};
         \draw[->,thick] (0.95,1.5)--(0.78,1.75);
         \node at (0.6,1.83) {\small $D_4$};
         \node at (1.15,1.45) {\small $D_5$};
         \draw[->,thick](1.42,1.53)--(1.7,1.2);
         \node at (1.83,1.05) {\small $D_6$};
         \node at (1.75,1.7) {\small $D_7$};

         \draw[thick,dashed] (2.5,0.2)--(2.5,0.6);
         \node at (3.3,0.3) {dividing lines.};
\end{tikzpicture}
\caption{The partition of a cut element: $K\cap\Omega^n_\eta=\bigcup_{i=1}^7D_i$.
Each sub-region $D_i$ is of either X-type or Y-type ($D_3$ is of Y-type and the others are of X-type). Each dividing line has at least one endpoint being a marker (green point) or an intersection point of $\Gamma^n_\eta\cap\partial K$ (blue point).} 	\label{fig:cutElement}
\vspace{-7mm}
\end{figure}

\subsection{The well-posedness of \eqref{weak-nh}}

Now we prove that the discrete problem \eqref{weak-nh} has a unique solution in each time step.
First we make a mild assumption on the finite element mesh. It can be satisfied
if $\Ct_h$ is fine and the deformation of the domain is moderate.
\vspace{1mm}

\begin{center}
\fbox{\parbox{0.98\textwidth}{
\begin{assumption}\label{ass-3}
There exist an integer $I>0$ and a $\gamma>0$ which are
independent of $h$ and $\tau$, such that, for any $K\in\Ct_{h,B}^n$,
one can find at most $I$ elements $\{K_j\}_{j=1}^I\subset\Ct^n_h$
satisfying $K_1=K$, $K_{j-1}\cap K_j\in\Ce^n_{h,B}$ for $1< j\le I$,
and that $K_I\cap\Omega^n_{\eta}$ contains a disk of radius $\gamma h$
  (see Fig.~\ref{fig:KI}).
\end{assumption}}}
\end{center}

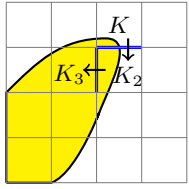
\begin{figure}[http!]
	\centering
\begin{tikzpicture}[scale =1.2]
\draw [thick,fill=yellow] (0,0)--(0,0.5*2) to [out=45,in=180](0.5*2+0.1,0.5*3+0.1)
		to [out=0, in=70] (1,0.7) to [out=70,in=20] (0.5,0)--(0,0);
		\node[align=right] at (1.25,1.75) {\small $K$};
		\node[align=right] at (1.35,1.18) {\small $K_2$};
		\node[align=right] at (0.7,1.2) {\small $K_3$};

		\draw[blue,very thick] (1,1.5)--(1.5,1.5);
		\draw[blue,very thick] (1,1.5)--(1,1);


		\draw [step =0.5cm,gray,thin] (0,0) grid (2cm,2cm);
		\draw[thick,->] (1.35,1.6)--(1.35,1.35);
		\draw[thick,->] (1.1,1.25)--(0.85,1.25);
\end{tikzpicture}
\caption{An illustration of Assumption~\ref{ass-3} for $I=3$: $K_3$
  contains a disk of radius $\gamma h$.}
\label{fig:KI}
\vspace{-5mm}
\end{figure}

To prove the well-posedness of \eqref{weak-nh}, we define the mesh-dependent norms
\begin{align*}
  &\TN{v}_{\Omega^n_{\eta}}  = \big(\SNHone[\Omega^n_{\eta}]{v}^2
    +h^{-1}\NLtwo[\Gamma^n_{\eta}]{v}^2
    + h\NLtwo[\Gamma^n_{\eta}]{\partial_{\Bn}v}^2\big)^{1/2} ,\\
  &\TN{v}_{\Ct^n_h}  = \big(\SNHone[\Omega^n_{\eta}]{v}^2
    +\mathscr{J}^n_0(v,v) +\mathscr{J}^n_1(v,v)\big)^{1/2} ,\\
  &\N{v}_{*,\Omega_h^n} =  \big(\SNHone[\Omega_h^n]{v}^2
    +h^{-1}\NLtwo[\Gamma^n_{\eta}]{v}^2  \big)^{1/2}.
\end{align*}

\begin{lemma} \label{lem:norm-eq}
Let Assumption~\ref{ass-3} be satisfied. Then for any $v_h\in V_h^{n}$,
\begin{align}	
\NLtwo[\Omega_h^n]{v_h}^2 &\,\lesssim \NLtwo[\Omega^n_{\eta}]{v_h}^2
      +h^2\mathscr{J}_1^n(v_h,v_h), 	\label{norm-eq0}\\
\TN{v_h}_{\Omega^n_{\eta}} &\,\lesssim \TN{v_h}_{\Ct^n_h}
    \eqsim \N{v_h}_{*,\Omega_h^n}.  \label{norm-eq1}
\end{align}
\end{lemma}
\begin{proof}
For each $K\in\Ct^n_{h,B}$, by Assumption~\ref{ass-3},
there exist (at most) $I$ elements $K_1=K,K_2,\cdots,K_I$
such that $E_j=K_{j-1}\cap K_j\in\Ce^n_{h,B}$, $1<j\le I$, and that
$K_I\cap\Omega^n_{\eta}$ contains a disk of radius $\gamma h$.
By \cite[Lemma~5.1]{mas14}, we have
\ben
\NLtwov[K_{j-1}]{\nabla^\mu v_h}^2  \lesssim  \NLtwov[K_{j}]{\nabla^\mu v_h}^2
  + \sum_{l=1}^{4}h^{2(l-\mu)+1}\int_{E_j}\jump{\partial_\Bn^l v_h}^2,
  \quad \mu =0,1.
\een
Since $K_I\cap\Omega^n_{\eta}$ contains a disk
of radius $\gamma h$, the norm equivalence shows
\begin{eqnarray}
    \NLtwov[K]{\nabla^\mu v_h}^2
    \lesssim  \NLtwov[K_I\cap\Omega^n_{\eta}]{\nabla^\mu v_h}^2
               + \sum_{j=2}^{I}\sum_{l=1}^{4} h^{2(l-\mu)+1}
               \int_{E_j}\jump{\partial_\Bn^l v_h}^2 ,
               \quad \mu=0,1. \label{ieq:vh0}
  \end{eqnarray}
Take the sum of \eqref{ieq:vh0} over all $K\in\Ct^n_{h,B}$.
Letting $\mu=0$ shows \eqref{norm-eq0}
and letting $\mu=1$ shows
\begin{equation}\label{vh-H1}
\SNHone[\Omega^n_h]{v_h}^2 \lesssim
    \SNHone[\Omega^n_{\eta}]{v_h}^2
    +\mathscr{J}_1^n(v_h,v_h) .
\end{equation}
Then we have $\N{v_h}_{*,\Omega_h^n}\lesssim \TN{v_h}_{\Ct^n_h}$.
Inverse estimates also yield
$\mathscr{J}_1^n(v_h,v_h)\lesssim \SNHone[\Omega_h^n]{v_h}^2$.
This shows
$\TN{v_h}_{\Ct^n_h}\eqsim \N{v_h}_{*,\Omega_h^n}$.
  The proof of
  $\TN{v_h}_{\Omega^n_{\eta}}\lesssim \N{v_h}_{*,\Omega_h^n}$
  is similar.
\end{proof}

\begin{theorem}\label{thm:exist}
Let Assumption~\ref{ass-3} be satisfied and suppose $\gamma_0$ is large enough.
Problem \eqref{weak-nh} has a unique solution $u^n_h\in V_h^{n}$ in each time step.
\end{theorem}
\begin{proof}
From \cite[Lemma~1 and Appendix A]{guz18}, we have the trace inequality, for any $K\in \Ct_h^n$,
\begin{align}	\label{ieq:trace0}
    \NLtwo[\partial K]{v} + \NLtwo[K\cap \Gamma_{\eta}^n]{v}\lesssim
    h^{-1/2}\NLtwo[K]{v}+h^{1/2}\SNHone[K]{v},
    \quad \forall\,v\in H^1(K).
\end{align}
Using the Cauchy-Schwarz inequality, it is standard to show the coercivity and continuity of $\mathscr{A}^n_h$: \begin{align}\label{eq:Ah}
\mathscr{A}_h^n(v_h,v_h) \gtrsim \TN{v_h}_{\Ct^n_h}^2, \quad
\SN{\mathscr{A}_h^n (u_h,v_h)} \lesssim
\TN{u_h}_{\Ct^n_h}\TN{v_h}_{\Ct^n_h},\quad
    \forall\,u_h,v_h\in {V^n_h}.
\end{align}
The details are omitted here.
The Lax-Milgram lemma shows that \eqref{weak-nh} has a unique solution.
\end{proof}

\section{The stability of numerical solutions}
\label{sec:well_posed}

The core difficulty in proving stability and convergence
is the fact that $U^{n-i,n}_h=u^{n-i}_h\circ\BX^{n,n-i}_\tau\notin
V^n_h$; we overcome this difficulty by introducing a modified Ritz projection operator
that projects $U^{n-i,n}_h$ into $V^n_h$.

\subsection{The modified Ritz projection}\label{sec:modified Ritz projection}

It is easy to see that
\begin{equation}\label{eq:Y}
U^{n-i,n}_h\in Y(\Omega_\eta^{n}) := \big\{v\in \Hone[\Omega^n_{\eta}]:
\TN{v}_{\Omega^n_{\eta}}<\infty\big\},\quad 0\le i\le 4.
\end{equation}
The modified Ritz projection operator
$\Cp^n_h:Y(\Omega^n_{\eta})\to V^n_h$ is defined as
\begin{align}	\label{eq:Pnh}
  \mathscr{A}^n_h(\Cp^n_hw,v_h)  =  a^n_h(w,v_h), \quad
  \forall\, v_h\in V^n_h,
\end{align}
where $a^n_h(w,v):=\mathscr{A}^n_h(w,v) - \mathscr{J}_1^{n}(w,v)$.
By \eqref{eq:Ah}, Lemma~\ref{lem:norm-eq}, and standard techniques for finite element error estimates,
it is easy to prove the results
\begin{align}
&\TN{\Cp^n_h v}_{\Omega^n_{\eta}} + \TN{\Cp^n_hv}_{\Ct^n_h}
    \lesssim \TN{v}_{\Omega^n_{\eta}},\quad \forall\,v\in Y(\Omega_\eta^n),
    \label{ieq:Rh-stab} \\
&\TN{w-\Cp^n_h w}_{\Omega^n_{\eta}}+ \TN{w-\Cp^n_h w}_{\Ct^n_h}
    +\dN{w-\mathcal{P}^n_h w}_{*,\Omega^n_h}
\lesssim  h^4\SN{w}_{H^5(D)},\quad \forall\,w\in H^5(D).
    \label{ieq:Ritz-1}
\end{align}
Below we only prove the $L^2$-error estimates by duality argument.

\begin{lemma}\label{lem:Ritz-0}
Suppose Assumption~\ref{ass-3} holds. Then
  \begin{align}
    &\NLtwo[\Omega^n_{\eta}]{v-\Cp^n_h v}
      \lesssim  h \TN{v}_{\Omega^n_{\eta}}, \quad
      \forall\, v\in Y(\Omega^n_{\eta}), 	\label{Rw-0} \\
    &\NLtwo[\Omega^n_{\eta}]{v-\Cp^n_h v}
      \lesssim  h^5\SN{v}_{H^5(\Omega^n_{\eta})}, \quad
      \forall\, v\in H^5(\Omega^n_{\eta}).
      \label{Rw-1}
  \end{align}
\end{lemma}
\begin{proof}
Consider the auxiliary problem
  \begin{equation}\label{eq: auxiliary problem}
    -\Delta z =v- \Cp^n_h v \quad
    \text{in}\;\; \Omega^n_{\eta},\qquad
    z = 0 \quad\text{on} \;\;\Gamma^n_{\eta}.
  \end{equation}
By Theorem~\ref{lem:bdr}, $\Gamma^n_{\eta}$ is $C^2$-smooth
and its parametrization satisfies $\N{\chibf_n}_{\BC^2([0,L])}\lesssim 1$.
The regularity result for elliptic equations yields
$\N{z}_{H^2(\Omega^n_{\eta})} \le C\NLtwo[\Omega^n_{\eta}]{v-\Cp^n_h v}$.

Multiplying both sides of the equation with $v-\Cp^n_h v$ and
  integrating by parts, we have
  \begin{equation}\label{eq:w1}
    \NLtwo[\Omega^n_{\eta}]{v-\Cp^n_h v}^2
    = \int_{\Omega^n_{\eta}}\nabla z\cdot\nabla (v-\Cp^n_h v )
    -\int_{\Gamma^n_{\eta}} \frac{\partial z}{\partial\Bn} (v-\Cp^n_h v)
    = a_h(v-\Cp^n_h v,z).
  \end{equation}
Let $\tilde z\in H^2(D)$ be the Sobolev extension of $z$
to the exterior of $\Omega^n_{\eta}$  \cite{ste70}.
There is a constant $C$ depending only on $\Omega^n_{\eta}$ such that
$\N{\tilde z}_{H^2(D)} \le C\N{z}_{H^2(\Omega^n_{\eta})}
  \le C  \NLtwo[\Omega^n_{\eta}]{v-\Cp^n_hv}$.
Let $\tilde z_h$ be the linear Scott-Zhang interpolation
of $\tilde z$ on the mesh $\Ct_h$. From \cite{sco94} and using \eqref{ieq:trace0}, we have
  \begin{equation}\label{eq:z-err}
    \TN{\tilde z-\tilde z_h}_{\Omega^n_{\eta}}^2
    +\sum_{E\in\Ce_{h,B}^n}
    h\int_E \jump{\partial_\Bn(\tilde z-\tilde z_h)}^2
    \lesssim h^2\SN{\tilde z}_{H^{2}(D)}^2
    \lesssim h^2\NLtwo[\Omega^n_{\eta}]{v-\Cp^n_h v}^2.
  \end{equation}
Inserting the equality $a_h(v-\Cp^n_hv,\tilde z_h) = \mathscr{J}_1^n(\Cp^n_h v,\tilde z_h)$ into \eqref{eq:w1} shows
  \begin{equation} \label{eq:ah-err-0}
    \NLtwo[\Omega^n_{\eta}]{v-\Cp^n_h v}^2
    = a_h(v-\Cp^n_h v,\tilde z-\tilde z_h)
    +\mathscr{J}_1^n(\Cp^n_h v,\tilde z_h).
  \end{equation}
Applying \eqref{ieq:Rh-stab} and \eqref{eq:z-err} shows that
  \begin{equation}\label{ah-err}
    a_h(v-\Cp^n_h v,\tilde z-\tilde z_h)
    \lesssim \TN{v-\Cp^n_h v}_{\Omega^n_{\eta}}
    \TN{\tilde z-\tilde z_h}_{\Omega^n_{\eta}}
    \lesssim h\TN{v}_{\Omega^n_{\eta}}
    \NLtwo[\Omega^n_{\eta}]{v-\Cp^n_h v}.
  \end{equation}
For each $E\in\Ce^n_{h,B}$, $\tilde z_h|_E$ is a linear function of either $x$ or $y$.
By \eqref{eq:z-err} and \eqref{ieq:Rh-stab}, we have
  \begin{align}
    \SN{\mathscr{J}_1^n(\Cp^n_h v,\tilde z_h)}
    = \Big|\sum_{E\in \Ce_{h,B}^n}
         h\int_E \jump{\partial_{\Bn}(\Cp^n_h v)}
         \jump{\partial_\Bn(\tilde z_h-\tilde z)}\Big|
    \lesssim h\TN{v}_{\Omega^n_{\eta}}
		\NLtwo[\Omega^n_{\eta}]{v-\Cp^n_h v}.
		\label{eq:ah-err-1}
  \end{align}
We get \eqref{Rw-0} by inserting \eqref{ah-err} and \eqref{eq:ah-err-1} into \eqref{eq:ah-err-0}.

Next let $\tilde{v}\in H^5(D)$ be the Sobolev extension of $v\in H^5(\Omega^n_{\eta})$.
Then $\jump{\partial_\Bn \tilde v}=0$ on each $E\in \Ce_{h,B}^n$.
Applying \eqref{eq:ah-err-0}, \eqref{eq:z-err}, and \eqref{ieq:Ritz-1} sequentially,
we find that
  \begin{align*}
    \NLtwo[\Omega^n_{\eta}]{v-\Cp^n_h v}^2
    \le\,& \TN{v-\Cp^n_h v}_{\Omega^n_{\eta}}
           \TN{\tilde z-\tilde z_h}_{\Omega^n_{\eta}}
           + \Big|\sum_{E\in\Ce_{h,B}^n} h\int_E
           \jump{\partial_{\Bn}(\tilde v-\Cp^n_h \tilde v)}
           \jump{\partial_\Bn(\tilde z-\tilde z_h)}\Big|\\
    \lesssim\, &  h^5 \SN{\tilde v}_{H^5(D)}
                 \NLtwo[\Omega^n_{\eta}]{v-\Cp^n_h v} .
  \end{align*}
  This finishes the proof.
\end{proof}

\subsection{The stability}

Now we are ready to prove the stability of numerical solutions. \vspace{1mm}

\begin{theorem} \label{thm:uh-stab}
Suppose Assumptions~\ref{ass-2} and \ref{ass-3} hold and that
the penalty parameter $\gamma_0$ is large enough in $\mathscr{A}^n_h$. Assume $\gamma_0\tau\le 1$ and $h = O(\tau)$. Then for $4\le m\le N$,
  \begin{align}\label{uh-stab}
    \NLtwo[\Omega^n_{\eta}]{u^m_h}^2
    +\sum_{n=4}^m\tau\TN{u_h^n}^2_{\Ct^n_h}
    \lesssim\,& \sum_{n=4}^m
		\tau\NLtwo[\Omega^n_\eta]{f^n}^2
		+\sum_{i=0}^3\left(
		\NLtwo[\Omega^i_\eta]{u^i_h}^2 +
		\tau\NHone[\Omega^i_h]{u^{i}_h}^2\right).
  \end{align}
\end{theorem}
\begin{proof}
Write $\tilde{U}^{n-i,n}_h=\Cp^n_h(U^{n-i,n}_h)\in V^n_h$ for convenience. We choose $v_h=2u^n_h -\tilde U^{n-1,n}_h$ as a test function in \eqref{weak-nh}. The telescope formula of BDF-4 (see \cite{liu13}) shows that
  \begin{equation}
    \sum_{i=1}^5 \NLtwo[\Omega^n_{\eta}]{\Psi_i^n}^2 -
    \sum_{i=1}^4 \NLtwo[\Omega^n_{\eta}]{\Phi_i^n}^2
    + \tau \mathscr{A}^n_h(u_h^n,2u_h^n-\tilde U^{n-1,n}_h)
    = A^n_1 +A^n_2,   \label{eq:stab0}
  \end{equation}
where $A^n_1= \tau \big(f^n,2u^n_h -\tilde U^{n-1,n}_h\big)_{\Omega^n_{\eta}}$,
$A_2^n=\sum_{i=0}^4 \lambda_i\big(U_h^{n-i,n}, \tilde U^{n-1,n}_h - U^{n-1,n}_h\big)_{\Omega^n_{\eta}}$, and
\begin{align*}
\Psi_i^n =\sum\limits_{j=1}^{i} c_{i,j} U_h^{n+1-j,n}, \qquad
\Phi_i^n = \sum\limits_{j=1}^{i} c_{i,j} U_h^{n-j,n}
=\Psi_i^{n-1}\circ \BX_{\tau}^{n,n-1}.
\end{align*}
The coefficients $c_{i,j}$ are real and given in \cite[Table~2.2]{liu13}.
We extend the discrete solutions to $D$ according to \eqref{extension} such that $u^n_h\in V_h$. By Lemma~\ref{lem:Lh} , $\Phi_i^n$ can be estimated as follows
\begin{align}\label{stab-est0}
\NLtwo[\Omega_{\eta}^n]{\Phi_i^n}^2 \le (1+C\tau)
\NLtwo[\Omega_{\eta}^{n-1}]{\Psi_i^{n-1}}^2
+ C\tau \sum_{j=1}^i\dN{u_h^{n-j}}_{\Ltwo[\Omega_h^{n-j}]}^2 .
\end{align}

Next we summarize from Lemmas~\ref{lem:uX} and \ref{lem:Uh} that, for $\mu=0,1$,
\begin{align}
&\dN{U^{n-j,n}_h}^2_{\Ltwo[\Gamma^n_\eta]}
 \le (1+C\tau)\dN{u^{n-j}_h}^2_{\Ltwo[\Gamma^{n-j}_\eta]}
 +C\tau^6h^{-2}\dN{u^{n-j}_h}^2_{\Hone[\Omega^{n-j}_{h}]}, \label{normL2-Gamma}\\
&\big|U^{n-j,n}_h\big|^2_{\Hone[\Gamma^n_\eta]}
 \lesssim h^{-1}\dN{u^{n-j}_h}^2_{H^1(\Omega_h^{n-j})}, \label{normH1-Gamma} \\
&\dN{\nabla^\mu U^{n-j,n}_h}^2_{\Ltwov[\Omega^{n}_\eta]}
\le (1+C\tau) \dN{\nabla^\mu u^{n-j}_h}^2_{\Ltwov[\Omega^{n-j}_\eta]}
+C\tau^6h^{-1}\dN{\nabla^\mu u^{n-j}_h}^2_{\Ltwov[\Omega^{n-j}_h]}. \label{eq:stab uh back1}
\end{align}
Inserting \eqref{normL2-Gamma}--\eqref{eq:stab uh back1} into
\eqref{norm-eq1} yields $\tN{U^{n-j,n}_h}_{\Omega^n_\eta}\lesssim
\tN{u^{n-j}_h}_{\Ct^{n-j}_h}$.
Remember that $\gamma_0\tau\le 1$ and $h = O(\tau)$. Using the definition of $\Cp^n_h$ and the Cauchy-Schwarz inequality,
we have
\begin{align}
\mathscr{A}^n_h(u_h^n,2u_h^n-\tilde U^{n-1,n}_h)
    =\,&2\mathscr{A}^n_h(u_h^n,u_h^n)
         -a^n_h(u_h^n, U^{n-1,n}_h)  \notag \\
\ge\,& \frac{3}{2}\TN{u_h^n}^2_{\Ct_h^n}
    -\frac{3\gamma_0}{5h}\NLtwo[\Gamma^{n}_\eta]{u^{n}_h}^2
    -C\gamma_0^{-1}\SNHone[\Omega_h^n]{u^n_h}^2
    -\frac{1}{2}\big|U^{n-1,n}_h\big|_{\Hone[\Omega^{n}_\eta]}^2  \notag \\
    &- C\gamma_0^{-1}\NHone[\Omega_h^{n-1}]{u^{n-1}_h}^2
     -\frac{3\gamma_0}{5h}\dN{U^{n-1,n}_h}_{\Ltwo[\Gamma^{n}_\eta]}^2 \notag \\
\ge\,&(0.9-C\gamma_0^{-1})\TN{u_h^n}^2_{\Ct_h^n}
    -(0.6+C\gamma_0^{-1})\TN{u^{n-1}_h}_{\Ct^{n-1}_h}^2. \label{stab-est1}
  \end{align}
By \eqref{Rw-0} and $\tN{U^{n-1,n}_h}_{\Omega^n_\eta}\lesssim
\tN{u^{n-1}_h}_{\Ct^{n-1}_h}$, the right hand side of \eqref{eq:stab0} satisfies
\begin{align}
A^n_1 \le\,& 2\tau\NLtwo[\Omega^n_\eta]{f^n}^2
		+\tau\NLtwo[\Omega^n_\eta]{u^n_h}^2
		+2\tau\dN{U^{n-1,n}_h}_{\Ltwo[\Omega^n_\eta]}^2
		+C \tau^3\tN{U^{n-1,n}_h}_{\Omega^n_\eta}^2, \label{stab-est2}\\
A^n_2 \le \,& C\tau\sum_{j=0}^4\dN{U^{n-j,n}_h}_{\Ltwo[\Omega^n_{\eta}]}^2
+\frac{\tau}{10}\tN{u^{n-1}_h}_{\Ct^{n-1}_h}^2.  \label{stab-est3}
  \end{align}

Substituting \eqref{stab-est0} and \eqref{stab-est1}--\eqref{stab-est3} into \eqref{eq:stab0},
we can find two positive constants $C_0$ and $C_1$ which are independent of $\eta$, $\tau$, and $h$, such that
\begin{align}
&\sum_{i=1}^4\Big[\NLtwo[\Omega^n_{\eta}]{\Psi^{n}_i}^2
-\NLtwo[\Omega^{n-1}_{\eta}]{\Psi_i^{n-1}}^2\Big]
+0.9\tau \TN{u_h^n}^2_{\Ct_h^n}  -0.7\tau\TN{u^{n-1}_h}_{\Ct^{n-1}_h}^2\notag \\
\le\,& C_0\tau\sum_{j=0}^4 \dN{u^{n-j}_h}_{\Ltwo[\Omega_\eta^{n-j}]}^2
+C_1\gamma_0^{-1}\tau \sum_{j=0}^4 \tN{u^{n-j}_h}_{\Ct^{n-j}_h}^2
+ 2\tau \NLtwo[\Omega^n_\eta]{f^n}^2.
      \label{eq:stab1}
\end{align}
Taking the sum of \eqref{eq:stab1} over $4\le n\le m$ and letting
$5C_1\gamma_0^{-1}\le 0.1$, we obtain
\begin{align}
\sum_{i=1}^4\NLtwo[\Omega_\eta^m]{\Psi_i^m}^2
+\tau\sum_{n=4}^m\TN{u_h^n}^2_{\Ct_h^n}
\lesssim \sum_{j=0}^3\tau\tN{u_h^{j}}^2_{\Ct_h^{j}}
+\sum_{n=0}^m\tau\left(\NLtwo[\Omega^n_\eta]{u^n_h}^2
+\NLtwo[\Omega^n_\eta]{f^n}^2\right).
\end{align}
Since $\Psi_1^m= 0.06u^m_h$ by \cite[Table~2.2]{liu13},
the proof is finished by using Gronwall's inequality.
\end{proof}

\section{A priori error estimates}
\label{sec:err}

The purpose of this section is to establish the a priori error estimates
for finite element solutions.

\subsection{Extended solution}
Since $\Omega^n_\eta\backslash\Omega_{t_n}\ne \emptyset$ in general,
we follow Lehrenfeld and Olshanskii \cite{leh19} to
extend the exact solution $u$ to the exterior of $\Omega_t$.
For convenience, we define
$Q_T=\left\{(\Bx,t):\Bx\in\Omega_t, t\in [0,T]\right\}$ and
\ben
L^\infty(0,T;H^m(\Omega_t))= \Big\{v\in\Ltwo[Q_T]:
\esssup_{t\in [0,T]}\N{v(\BX(t;0,\cdot),t)}_{H^m(\Omega_0)}
<+\infty\Big\}, \quad m\ge 0.
\een
By \cite[Chapter~6]{ste70}, there is an extension operator
$\mathsf{E}_0$: $H^5(\Omega_0)\to H^5(\bbR^2)$
such that
\begin{equation*}
	\left(\mathsf{E}_0 w\right)|_{\Omega_0} =w,\quad
	\|\mathsf{E}_0 w\|_{H^5(\bbR^2)}\lesssim \|w\|_{H^5(\Omega_0)},
    \quad \forall\,w\in H^5(\Omega_0).
\end{equation*}
Since $\BX(t;0,\cdot)$ is one-to-one, its inverse is denoted by $\BX(0;t,\cdot)$.
Then $\Omega_0=\BX(0;t,\Omega_t)$. We can define
an extension operator from $H^5(\Omega_t)$ to $H^5(\bbR^2)$ as
\ben
\mathsf{E}_t w :=
\big[\mathsf{E}_0\big(w\circ\BX(t;0,\cdot)\big)\big]\circ \BX(0;t,\cdot).
\een
The global extension operator
$\mathsf{E}$: $L^\infty(0,T;H^5(\Omega_t))\to L^\infty(0,T;H^5(\bbR^2))$
is defined as
\ben
(\mathsf{E}v)(t)=\mathsf{E}_t(v(t))\qquad \forall\, t\in [0,T].
\een
By \eqref{eq:Jnni} and arguments similar to \cite{leh19},
we have the stability estimates for the extension operator
\begin{align}\label{ieq:extension}
\begin{cases}
\|\mathsf{E} v\|_{H^5(\bbR^2\times [0,T])} \le C \|v\|_{H^5(Q_T)},
    \vspace{1mm}\\
\|(\mathsf{E} v)(t)\|_{H^{m}(\bbR^2)} \le
    C \|v(t)\|_{H^{m}(\Omega_t)}, \quad 1\leq m\leq 5,
     \vspace{1mm} \\
\|\partial_t(\mathsf{E} v)(t)\|_{H^1(\bbR^2)}
\le C \left[\|v(t)\|_{H^2(\Omega_t)}
  +\|(\partial_t v)(t)\|_{H^1(\Omega_t)} \right],
\end{cases}
\end{align}
where the constant $C>0$ depends only on $\Omega_0$
and $\N{\Bw}_{\BC^4(\bbR^2\times [0,T])}$.

Let $u$ be the exact solution to \eqref{cd-model} and define $\tilde{u}=\mathsf{E} u$,
$u^n:=\tilde u(t_n)$, and $U^{m,n}:= u^m\circ \BX^{n,m}$.
Multiplying $\sum_{i=0}^4 \lambda_iU^{n-i,n} -\tau\Delta u^n$ with $v_h\in V^n_h$
and using integration by parts, we have
\begin{equation}\label{eq:weakU}
\frac{1}{\tau}\sum_{i=0}^4 \lambda_i\big(U^{n-i,n}, v_h\big)_{\Omega^n_{\eta}}
+a^n_h(u^n,v_h) =\int_{\Gamma_\eta^n} u^n\Big(\frac{\gamma_0}{h}v_h-\partial_{\Bn} v_h\Big)
    +(\tilde{f}^n+R^n,v_h)_{\Omega^n_{\eta}},
\end{equation}
where
$\tilde{f}^n =\frac{\partial\tilde u}{\partial t}(t_n)
+ \Bw(t_n)\cdot\nabla u^n-\Delta u^n$ and
$R^n= \tau^{-1}\sum_{i=0}^4 \lambda_iU^{n-i,n}- \frac{\partial\tilde u}{\partial t}(t_n)
- \Bw(t_n)\cdot\nabla u^n$.

\subsection{Error estimates}
Now we present the main theorem of this section. Suppose that
the exact solution $u$ and the source function $f$ satisfy
\begin{align*}
M_u:=\N{u}_{H^5(Q_T)}^2 +\N{u}_{L^\infty(0,T;H^5(\Omega_t))}^2
+\N{\partial_t u}_{L^\infty(0,T;H^1(\Omega_t))}^2
+\N{f}_{L^\infty(0,T;H^1(D))}^2  < \infty,
\end{align*}
and that the pre-calculated initial values satisfy
\begin{equation}\label{init-err}
\NLtwo[\Omega_\eta^i]{u^i-u_h^i}^2+\tau \tN{u^i-u_h^i}_{\Ct_h^i}^2
		\le C_0\tau^8, \qquad 0\le i\le 3.
\end{equation}

\begin{theorem}\label{thm:err}
Suppose the assumptions in Theorem~\ref{thm:uh-stab} hold.
Then for any $4\leq m\leq N$,
\begin{align*}
\NLtwo[\Omega^m_\eta]{u^m-u_h^m}^2
+\sum_{n=4}^{m}\tau \TN{u^{n}-u_h^{n}}_{\Ct_h^n}^2
\lesssim\,& (C_0+M_u)\tau^8 \markChange{.}
\end{align*}
\end{theorem}
\begin{proof}
Write $\rho^n := u^n - \Cp_h^n u^n$ and $\theta_h^n := \Cp_h^n u^n - u_h^n$.
By Lemma~\ref{lem:Ritz-0}, \eqref{ieq:Ritz-1}, and \eqref{ieq:extension}, we have
\begin{equation}\label{ieq:rhon}
\max_{1\le n\le N}\NLtwo[\Omega^n_{\eta}]{\rho^n}^2
+\sum_{n=1}^N\tau \TN{\rho^n}_{\Ct^n_h}^2
\lesssim  \tau^{8} \N{u}_{L^{\infty}(0,T;H^{5}(\Omega_t))}^2.
\end{equation}
It is left to estimate $\theta_h^n$ for $4\le n\le N$.
The arguments are similar to the proof of Theorem~\ref{thm:uh-stab}.
Due to \eqref{init-err}--\eqref{ieq:rhon} and \eqref{norm-eq0}, the pre-calculated initial values satisfy
\begin{equation}\label{init-err-disc}
\NLtwo[\Omega_h^i]{\theta_h^i}^2+\tau \tN{\theta_h^i}_{\Ct_h^i}^2
\lesssim \NLtwo[\Omega_\eta^i]{\theta_h^i}^2+\tau \tN{\theta_h^i}_{\Ct_h^i}^2
		\lesssim C_0\tau^{8}, \quad 0\le i\le3.
\end{equation}

Define $\Theta_h^{m,n}=\theta_h^m\circ\BX_\tau^{n,m}$ and
$\zeta^{m,n} =(u^{m}\circ\BX_\tau^{n,m}-u^{m}\circ\BX^{n,m}) -\rho^m\circ \BX_\tau^{n,m}$.
Subtracting \eqref{weak-nh} from \eqref{eq:weakU} and using \eqref{eq:Pnh},
we get
\begin{equation}\label{eq:err}
\sum_{i=0}^4\lambda_i (\theta_h^{n-i,n},v_h)_{\Omega_\eta^n}
+ \tau\mathscr{A}_h^n(\theta_h^n,v_h)
=\tau(g^n,v_h)_{\Omega_\eta^n} + \tau\ell_n(v_h),
\end{equation}
where
$g^n= R^n+\tau^{-1}\sum_{i=0}^4\lambda_i\zeta^{n-i,n}$ and
$\ell_n(v_h)=(\tilde{f}^n-f^n,v_h)_{\Omega^n_\eta}
+\int_{\Gamma_\eta^n}(\gamma_0h^{-1}v_h-\partial_{\Bn} v_h)u^n$.
Choosing $v_h^n= 2\theta_h^{n}-\Cp_h^n \Theta_h^{n-1,n}$ in \eqref{eq:err} and using
\eqref{uh-stab} and \eqref{init-err}, we obtain
\begin{align}\label{theta-stab}
\NLtwo[\Omega_\eta^m]{\theta_h^m}^2 +
\tau\sum_{n=4}^m \TN{\theta_h^n}_{\mathcal{T}_h^n}^2
\lesssim C_0\tau^8 + \sum_{n=4}^m\tau\left[\NLtwo[\Omega^n_\eta]{g^n}^2
+\SN{\ell_n(v_h^{n})}\right].
\end{align}
It suffices to estimate $g^n$ and $\ell_n(v_h^{n})$ on the right-hand side of \eqref{theta-stab}.

Applying Taylor's formula to $R^n$ yields
\begin{align*}
&\NLtwo[\Omega_\eta^n]{R^n}^2 =
		\int_{\Omega_\eta^n}\bigg|
		\sum_{i=1}^4 \frac{\lambda_i}{4!\tau}
		\int_{t_{n-i}}^{t_n} (t_n-t)^4
		\frac{\D^5}{ \D t^5}u(\BX(t;t_n,\Bx),t)\D t\bigg|^2\D\Bx
		\lesssim \tau^7\N{u}_{H^5(D\times (t_{n-4},t_n))}^2.
\end{align*}
Let $\Ci_h u^{n-i}\in V_h$ be the Scott-Zhang interpolation of $u^{n-i}$ and define
$\rho^{n-i}_h=\Ci_h u^{n-i}-\Cp^{n-i}_hu^{n-i}$. From \eqref{ieq:rhon}, we obtain
\ben
\N{\rho_h^{n-i}}_{\Ltwo[\Omega^{n-i}_h]} \le \N{\rho^{n-i}}_{\Ltwo[\Omega^{n-i}_h]}
+\N{u^{n-i}-\Ci_h u^{n-i}}_{\Ltwo[\Omega^{n-i}_h]}\lesssim h^5\|u^{n-i}\|_{H^5(D)}.
\een
Moreover, combining Lemma~\ref{lem:Uh} and Lemma~\ref{lem:norm-eq} shows
\begin{align*}
\NLtwo[\Omega^n_\eta]{\rho^{n-i}\circ \BX_\tau^{n,n-i}}
\lesssim\,& \NLtwo[\Omega^{n-i}_h]{u^{n-i}-\Ci_h u^{n-i}}
    +\NLtwo[\Omega^n_\eta]{\rho_h^{n-i}\circ \BX_\tau^{n,n-i}} \\
\lesssim\,& h^5\|u^{n-i}\|_{H^5(D)} + \dN{\rho_h^{n-i}}_{\Ltwo[\Omega^{n-i}_\eta]}
+\tau^{2.5}\NLtwo[\Omega^{n-i}_h]{\rho_h^{n-i}}\\
\lesssim\,& \tau^5\|u^{n-i}\|_{H^5(D)}.
\end{align*}
Using \eqref{eq:Xerr} and Lemma~\ref{lem:uX}, we get
$\NLtwo[\Omega_\eta^n]{\zeta^{n-i,n}}
\lesssim\tau^5\|u^{n-i}\|_{H^{5}(D)}$; and from \eqref{ieq:extension}, we have
\begin{align}\label{est-gn}
\sum_{n=4}^m\tau\NLtwo[\Omega^n_\eta]{g^n}^2 \lesssim \tau^8\big(
\N{u}_{L^\infty(0,T;H^5(\Omega_t))}^2 + \N{u}^2_{H^5(Q_T)}\big) .
\end{align}

Next we estimate $\ell_n(v_h^{n})$.
Using Lemma~\ref{lem:v-vh} and the identity $\tilde f^n=f^n$
in $\Omega_{t_n}$, we find that
\begin{equation}\label{ieq:fn}
\big|(\tilde{f}^n-f^n,v_h^n)_{\Omega^n_\eta}\big|
=\big|(\tilde{f}^n-f^n,v_h^n)_{\Omega^n_\eta\backslash \Omega_{t_n}}
\lesssim\tau^{4.5}\|\tilde{f}^n -f^n\|_{H^1(\Omega^n_\eta)}
\NLtwo[\Omega^n_\eta]{v_h^n}.
\end{equation}
Since $u^n=0$ on $\Gamma_{t_n}$, Taylor's formula and Lemma~\ref{lem:bdr} indicate
\begin{align*}
\NLtwo[\Gamma_{\eta}^n]{u^n}^2 =
\int_0^{L_0} (u^n\circ\chibf_n-u^n\circ\hat\chibf_n)^2
\SN{\chibf_n'}
\lesssim  \tau^{10}\SN{u^n}_{H^1(D)}^2.
\end{align*}
Using inverse estimate, we obtain
\begin{align}\label{ieq:ln}
\SN{\ell_n(v^n_h)} \lesssim \tau^{4.5}\|\tilde{f}^n -f^n\|_{H^1(\Omega^n_\eta)}
\NLtwo[\Omega^n_\eta]{v_h^n}
+\tau^4\N{u^n}_{H^1(D)}\NLtwo[\Gamma^n_\eta]{v_h^n}.
\end{align}

The techniques for estimating $v_h^n=2\theta_h^{n}-\Cp_h^n \Theta_h^{n-1,n}$ are similar to \eqref{stab-est2}. They use Lemma~\ref{lem:norm-eq}, \eqref{ieq:Rh-stab}, \eqref{Rw-1}, Lemma~\ref{lem:uX}, and Lemma~\ref{lem:Uh}. We omit the details and just present the results
\begin{align*}
\N{v_h^n}_{\Ltwo[\Omega^{n}_h]}^2
\lesssim\,& \N{\theta^n_h}_{\Ltwo[\Omega^{n}_\eta]}^2
+\N{\theta^{n-1}_h}_{\Ltwo[\Omega^{n-1}_\eta]}^2
+h^2\tN{\theta^{n}_h}^2_{\Ct^{n}_h}
+h^2\tN{\theta^{n-1}_h}_{\Ct^{n-1}_h}^2,   \\
\N{v_h^n}^2_{\Ltwo[\Gamma^n_\eta]}
\lesssim\,& h\mathscr{J}^n_0(\theta^n_h,\theta^n_h)
+h\tN{\Theta^{n-1,n}_h}_{\Omega^n_\eta}^2
\lesssim h\tN{\theta^{n}_h}_{\Ct^n_h}^2
+h\tN{\theta^{n-1}_h}_{\Ct^{n-1}_h}^2.
\end{align*}
Inserting the estimates into \eqref{ieq:fn} and \eqref{ieq:ln} yields
\begin{align}\label{est-ln}
\sum_{n=4}^m\tau\SN{\ell_n(v^n_h)} \lesssim \tau^8 M_u
+ \tau^2\sum_{n=3}^m\left(\N{\theta^n_h}_{\Ltwo[\Omega^{n}_\eta]}^2
+ \tN{\theta^{n}_h}_{\Ct^n_h}^2 \right).
\end{align}
Finally, we insert \eqref{est-gn}--\eqref{est-ln} into \eqref{theta-stab} and let $\tau$ be small enough. This yields
\begin{align}\label{theta-err}
\NLtwo[\Omega_\eta^m]{\theta_h^m}^2 +
\tau\sum_{n=4}^m \TN{\theta_h^n}_{\mathcal{T}_h^n}^2
\lesssim (C_0+M_u)\tau^8.
\end{align}
The proof is finished by combining \eqref{ieq:rhon} and \eqref{theta-err}.
\end{proof}

\section{Numerical experiments}
\label{sec:num}

Now we use two numerical experiments to verify the convergence order of the UCFEM. The exact solution
is set by $u(\Bx,t) = e^{-t}\sin(\pi x_1)\sin(\pi x_2)$.
To simplify the computation, we set the pre-calculated initial values by the exact solution, namely, $u^j_h=u(t_j)$, for $0\le j\le 3$.
Throughout the section, we set $\gamma_0=800$, $\gamma_1=1/\gamma_0$, and $\eta \leq 0.5\tau$.
The approximation error is measured with the quantity
$e^N=\big[\|u(\cdot,T)-u_h^{N}\|_{L^2(\Omega_\eta^N)}^2
+\sum_{n=4}^N\tau\SNHone[\Omega_\eta^n]{u-u_h^n}^2\big]^{1/2}$.

\subsection{A rotating elliptic ring}

This example is to test the robustness and optimal convergence of the UCFEM for rigid motions.
The velocity is set by $\Bw = (0.5-x_2,x_1-0.5)^\top$.
The initial domain $\Omega_0$ is an elliptical ring centered at $(0.5,0.5)$. The inner boundary is an ellipse with major axis equal to 0.22 and minor axis equal to 0.1. The outer boundary is an ellipse with major axis equal to 0.3 and minor axis equal to 0.15. The ring has rotated half a circle counterclockwise at $T=\pi$.

\begin{figure}[http!]
	\centering
\vspace{-2mm}
	\includegraphics[width=0.18\textwidth]{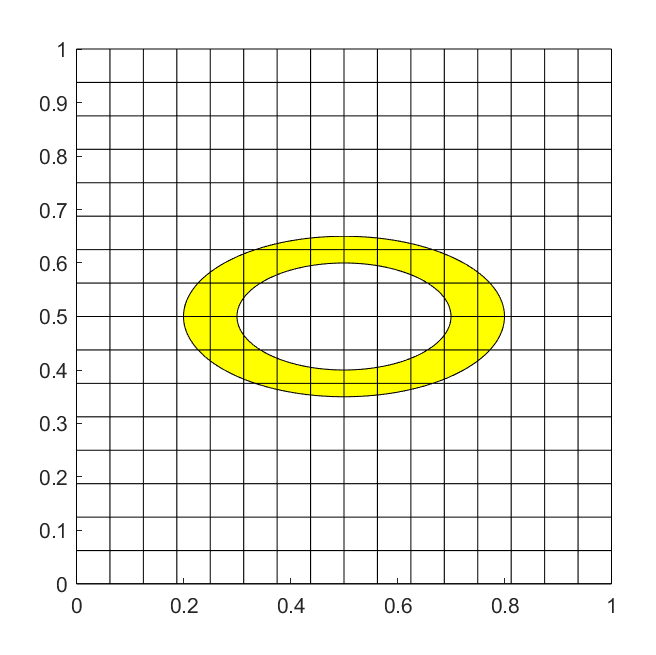}
	\includegraphics[width=0.18\textwidth]{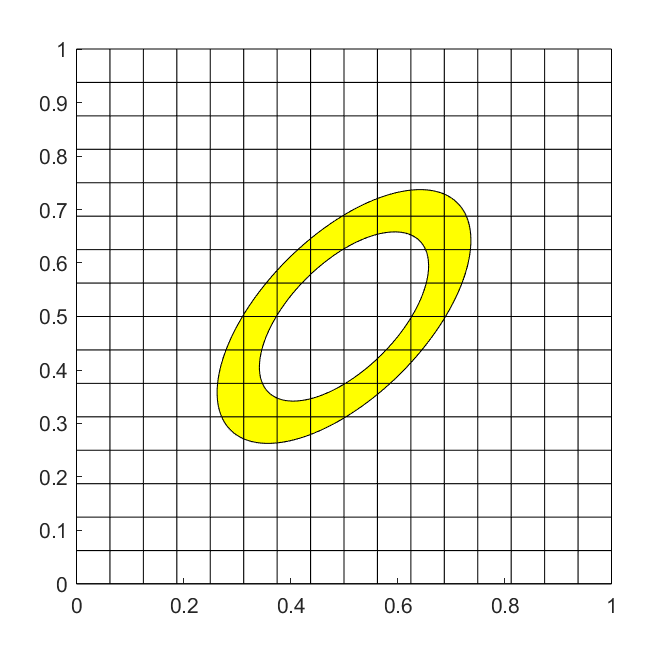}
	\includegraphics[width=0.18\textwidth]{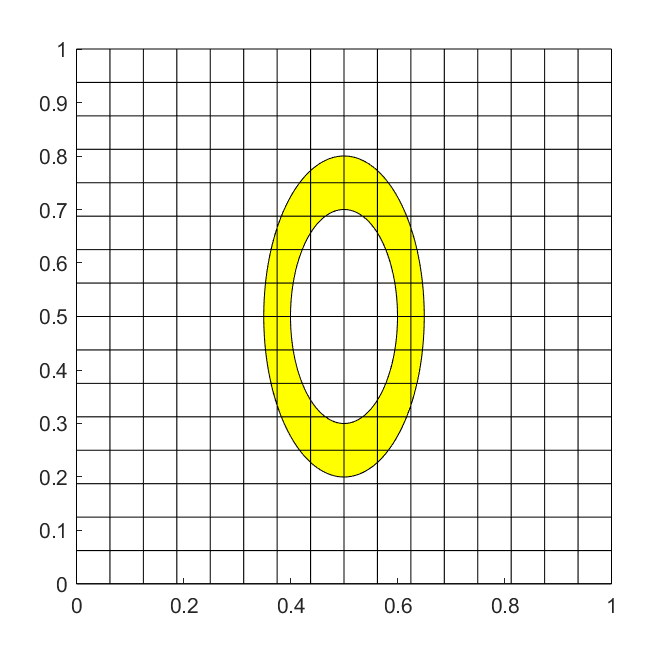}
	\includegraphics[width=0.18\textwidth]{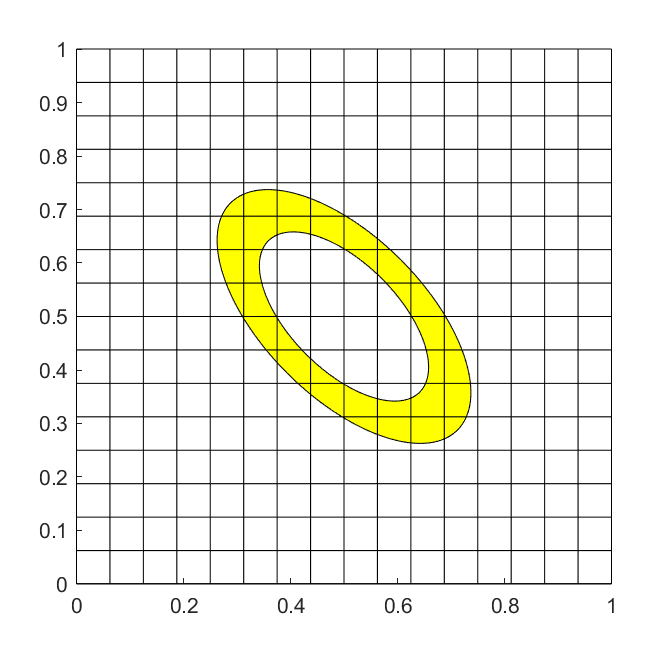}
\vspace{-2mm}
\caption{The approximate domains at $t_n=0,\;\pi/4,\;\pi/2$, and $3\pi/4$, respectively ($h=2^{-4}$).}	\label{fig:rotation} \vspace{-8mm}
\end{figure}

{\linespread{1.1}
\begin{table}[http!]
	\center
\vspace{-3mm}
\caption{Convergence orders for Example~1 ($h=\tau/\pi$).}
\label{tab:resul rotation} \vspace{-3mm}
\setlength{\tabcolsep}{3mm}
\begin{tabular}{ccc|ccc}
		\hline
$h$       &$e^N$     &order  &$h$       &$e^N$     &order \\  \hline
$2^{-4}$  &2.98e-06  &---      &$2^{-6}$  &1.67e-08  &3.86  \\
$2^{-5}$  &2.43e-07  &3.62   &$2^{-7}$  &1.09e-09  &3.93  \\  \hline
\end{tabular}
\vspace{-1mm}
\end{table}}

Since the rigid motion of the ring is an isometry,
in Algorithm~\ref{alg:spline} there is no need to redistribute the
markers on the boundary.
The interface tracking error only comes from cubic spline interpolations.
Our method can achieve good accuracy even on coarse meshes.
Table~\ref{tab:resul rotation} shows that the optimal convergence $e^N \sim\tau^4$ is obtained asymptotically.

\subsection{Vortex flow}

The second example is to test the robustness and optimal convergence of the UCFEM for severely deformed domains.
The initial domain $\Omega_0$ is the disk whose radius is 0.15 and center is $(0.5,\,0.75)$.
The driving velocity is set by
\ben
\Bw =\cos(\pi t/4) \left(\sin^2 (\pi x_1) \sin (2\pi x_2),
	-\sin^2 (\pi x_2 )\sin (2\pi x_1)\right)^\top.
\een
At time $T=2$, $\Omega_T$ is stretched into a snake-like domain (see Fig.~\ref{fig:Omegat}).
In this example, $\Gamma_{t}$ suffers a large deformation
and yields a local $C^1$-discontinuity at the final time. To guarantee the
high accuracy of the method, we use Algorithm~\ref{alg:spline} to track the boundary and adjust the set of markers dynamically in time.
Since the exact solution is analytic, we can still observe optimal convergence as shown in Tables~\ref{tab:conv}.

\begin{figure}[http!]
	\centering
	\includegraphics[width=0.18\textwidth]{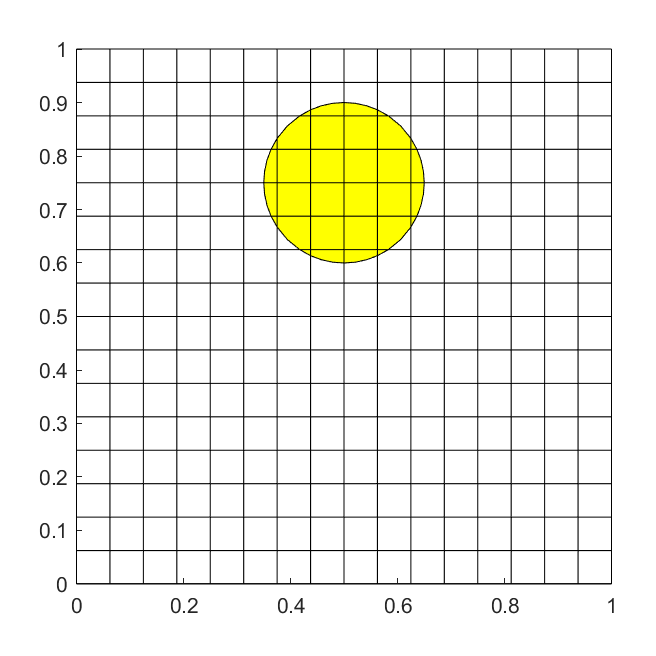}
	\includegraphics[width=0.18\textwidth]{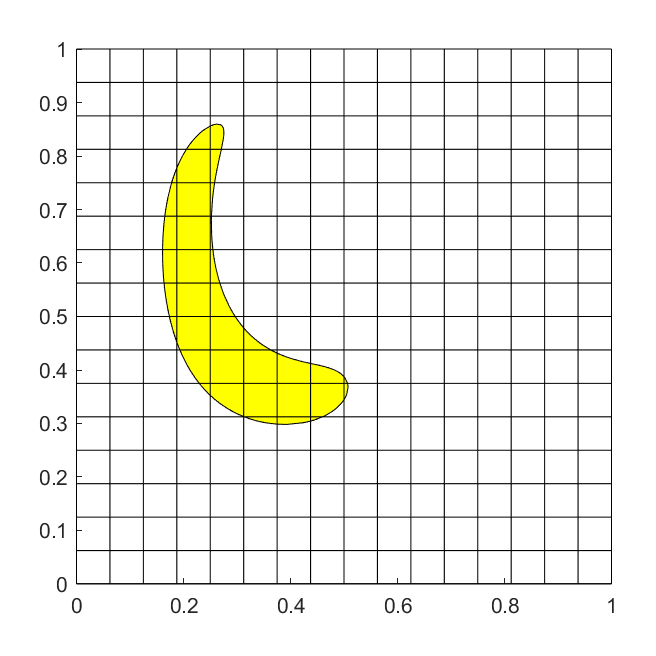}
	\includegraphics[width=0.18\textwidth]{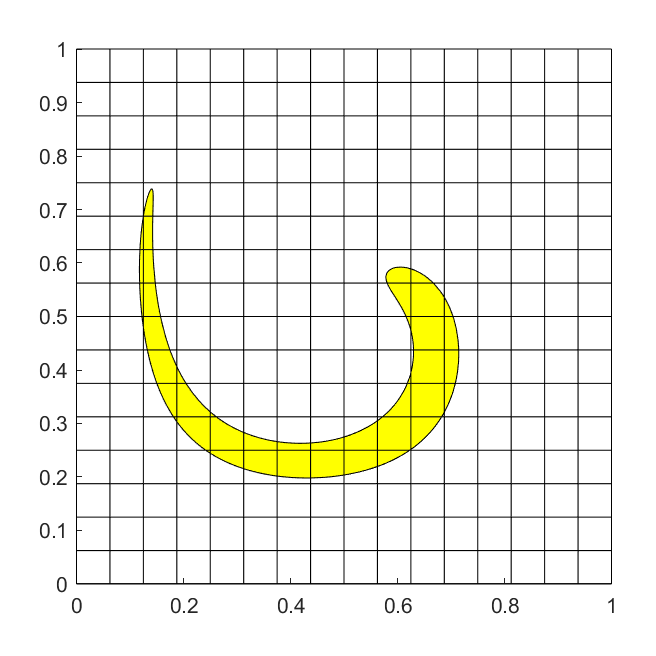}
	\includegraphics[width=0.18\textwidth]{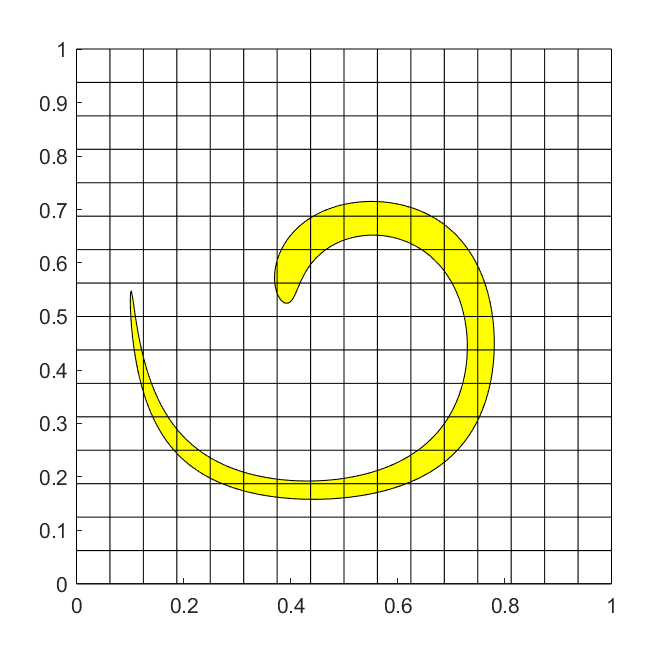}	
	\caption{The approximate domains $\Omega^n_\eta$ at $t_n=0,\;1/2,\;1$, and $2$, respectively ($h=2^{-4}$).}	\label{fig:Omegat}
\vspace{-10mm}
\end{figure}

{\linespread{1.1}
\begin{table}[http!]
	\center
\setlength{\tabcolsep}{3mm}
\vspace{-2mm}
\caption{Convergence orders for Example~2 ($h=\tau$).} \label{tab:conv}
\vspace{-3.5mm}
\begin{tabular}{ccc|ccc}
		\hline
$h$       &$e^N$     &order  &$h$       &$e^N$     &order  \\  \hline
$2^{-4}$  &2.43e-06  &---      &$2^{-6}$   &4.56e-09  &4.44   \\
$2^{-5}$  &9.90e-08  &4.62   &$2^{-7}$   &2.34e-10  &4.29   \\ \hline
\end{tabular}
\vspace{-5mm}
\end{table}}

\vspace{2mm}	

\begin{appendix}

\section{Estimates of $v\circ\BX^{n,n-i}$ and $v\circ\BX_\tau^{n,n-i}$} \label{sec:app}
In this appendix, we prove some useful estimates
for $v\circ\BX^{n,n-i}$ and $v\circ\BX^{n,n-i}_\tau$ with $0\le i\le 4$.

\begin{lemma}\label{lem:uX}
Suppose $h=O(\tau)=O(\eta^{4/5})$ and $\Omega\cup\BX^{n,n-i}_\tau(\Omega)\cup \BX^{n,n-i}(\Omega)\subset D$. There exists an $h_0>0$ such that, for any $h\in (0,h_0]$, $v\in H^1(D)$, and $v_h\in V_h$,
\begin{align}
 &\NLtwov[\Omega]{\nabla^\mu
 (v\circ\BX^{n,n-i}_\tau)}^2
 \le (1+C\tau)
 \NLtwov[\BX^{n,n-i}_\tau(\Omega)]{\nabla^\mu v}^2,
 \qquad \mu=0,1, \label{eq:v-X0}\\
 &\NLtwo[\Omega]{v\circ\BX^{n,n-i} -v\circ\BX^{n,n-i}_\tau}
  \lesssim \tau^{6}\SNHone[D]{v},    \label{eq:v-X1} \\
&\SNHone[\Gamma^n_\eta]{v_h\circ\BX_\tau^{n,n-i}}^2
 \lesssim h^{-1}\SNHone[\Omega^{n-i}_h]{v_h}^2,	\label{eq:v-X3} \\
&\NLtwo[\Gamma^n_\eta]{v_h\circ\BX_\tau^{n,n-i}}^2
 \le (1+C\tau)\NLtwo[\Gamma^{n-i}_\eta]{v_h}^2
 +C\tau^4\NHone[\Omega^{n-i}_h]{v_h}^2. \label{eq:v-X2}
\end{align}
\end{lemma}
\begin{proof}
Inequality \eqref{eq:v-X0} is obtained directly by changing variables of integration and using \eqref{eq:Jnni}. Inequality
\eqref{eq:v-X1} is a direct consequence of \eqref{eq:Xerr}. Moreover, \eqref{eq:v-X3} can be proven easily by using scaling arguments, norm equivalence, and the results in \eqref{eq:Jnni}.

To prove \eqref{eq:v-X2}, we note that
\begin{align*}
\NLtwo[\Gamma^n_\eta]{v_h\circ\BX_\tau^{n,n-i}}^2
= \int_0^{L_0} \SN{v_h\circ \chibf_{n-i}}^2\SN{\chibf_n'} + \int_{0}^{L_0}
\big(\SN{v_h\circ \BX_\tau^{n,n-i}\circ \chibf_{n}}^2-
	 \SN{v_h\circ \chibf_{n-i}}^2\big)
	 \SN{\chibf_{n}'}.
\end{align*}
Remember from \eqref{tdomain} that
$\min\limits_{\Bx\in\partial\Omega^{m}_h}\mathrm{dist}(\Bx,\Gamma^{m}_\eta) \ge h/2$ for any $m>0$. Then, for $h$ small enough,
\begin{equation}\label{domain-nni}
\BX_\tau^{n,n-i}(\Omega^{n}_\eta)\subset\Omega^{n-i}_h,\qquad
\BX_\tau^{n-i,n}(\Omega^{n-i}_\eta)\subset\Omega^{n}_h.
\end{equation}
By Theorem~\ref{thm:chi-err} and the relation $\eta=O(\tau^{5/4})$, it is easy to see that
\begin{equation}\label{err-chi}
\begin{cases}
\N{\chibf_n-\BX_\tau^{n-i,n}\circ \chibf_{n-i}}_{\BC([0,L_0])}
\lesssim \tau^{6},    \vspace{1mm}\\
\N{\BX_\tau^{n,n-i}\circ \chibf_n-\chibf_{n-i}}_{\BC([0,L_0])} \lesssim
\N{\chibf_n-\BX_\tau^{n-i,n}\circ \chibf_{n-i}}_{\BC([0,L_0])}\lesssim \tau^{6}.
\end{cases}
\end{equation}
With \eqref{eq:Jnni}, it is standard to derive $\SN{\chibf_{n-i}'} \le\SN{\chibf_{n}'}+C\tau$ and
$\SN{\chibf_n'}\SN{\chibf_{n-i}'}^{-1} \leq 1+C\tau$. So we have
\begin{align}\label{vh-est1}
\int_{0}^{L_0} \SN{v_h\circ \chibf_{n-i}}^2  \SN{\chibf_{n}'}
\le (1+C\tau) \int_{0}^{L_0} \SN{v_h\circ \chibf_{n-i}}^2  \SN{\chibf_{n-i}'}
=(1+C\tau)\NLtwo[\Gamma^{n-i}_\eta]{v_h}^2.
\end{align}
Using Taylor's formula, inequality \eqref{err-chi}, and norm equivalence on each element, we also have
\begin{align}
\int_0^{L_0}\left(
\SN{v_h\circ\BX_\tau^{n,n-i}\circ\chibf_n}^2
-\SN{v_h\circ\chibf_{n-i}}^2\right) \SN{\chibf_{n}'}
\lesssim \tau^{6} h^{-2}\NHone[\Omega^{n-i}_h]{v_h}^2.
\label{vh-est2}
\end{align}
We obtain \eqref{eq:v-X2} from \eqref{vh-est1} and \eqref{vh-est2}.
\end{proof}

\begin{lemma}\label{lem:Uh}
Suppose $h=O(\tau)=O(\eta^{4/5})$.
There is a constant $C$ independent of $\tau$ such that
\begin{equation*}
\dN{\nabla^\mu(v_h\circ\BX^{n,n-i}_\tau)}_{\Ltwov[\Omega^n_\eta]}^2 \le(1+C\tau)
\dN{\nabla^\mu v_h}_{\Ltwov[\Omega^{n-i}_\eta]}^2
+C\tau^5 \NLtwo[\Omega^{n-i}_h]{\nabla^\mu v_h}^2,\quad
\mu=0,1 .
\end{equation*}
\end{lemma}
\begin{proof}
Clearly $\Omega^n_\eta\subset \BX^{n-i,n}_\tau(\Omega^{n-i}_\eta)\cup D^n_\eta$ where
$D^n_\eta:=\Omega^n_\eta\backslash\BX^{n-i,n}_\tau(\Omega^{n-i}_\eta)$ is the narrow strip between $\Gamma^n_\eta$ and $\BX^{n-i,n}_\tau(\Gamma^{n-i}_\eta)$.
Similarly, $\BX^{n,n-i}_\tau(D^n_\eta)$ is the narrow strip between
$\BX^{n,n-i}_\tau(\Gamma^n_\eta)$ and $\Gamma^{n-i}_\eta$.
From \eqref{domain-nni} and \eqref{err-chi}, we have
$\BX^{n,n-i}_\tau(D^n_\eta)\subset\Omega^{n-i}_h$ and
$\hbox{area}\left[K\cap \BX^{n,n-i}_\tau(D^n_\eta)\right]
\lesssim \tau^{6}h$ for any $K\in\Ct_h$.
The scaling argument shows that
\begin{align}
\dN{v_h\circ\BX^{n,n-i}_\tau}^2_{\Ltwo[D^n_\eta]}
\lesssim \sum_{K\cap \BX^{n,n-i}_\tau(D^n_\eta)\ne\emptyset}
h\tau^{6} \NLinf[K]{v_h}^2
\lesssim h^{-1}\tau^{6}\NLtwo[\Omega^{n-i}_h]{v_h}^2.
\label{eq:areErrvh}
\end{align}
Note that $\nabla(v_h\circ\BX^{n,n-i}_\tau)=\bbJ^{n,n-i}_\tau(\nabla v_h)\circ\BX^{n,n-i}_\tau$.
Using \eqref{eq:v-X0}, we have
\begin{align*}
\dN{v_h\circ\BX^{n,n-i}_\tau}_{\Ltwo[\Omega^n_\eta]}^2
\le (1+C\tau) \dN{v_h}^2_{\Ltwo[\Omega^{n-i}_\eta]}
+ h^{-1}\tau^{6}\NLtwo[\Omega^{n-i}_h]{v_h}^2.
\end{align*}
The proof for the case of $\mu=1$ is similar.
\end{proof}

\begin{lemma}\label{lem:Lh}
Suppose the conditions in Lemma~\ref{lem:Uh} hold. Let $w_n=\sum_{i=1}^4v_i\circ\BX^{n,n-i}_\tau$ and $w_{n-1}=\sum_{i=1}^4v_i\circ\BX^{n-1,n-i}_\tau$ with $v_i\in V_h$.
There is a constant $C>0$ independent of $\tau$ such that
\begin{equation*}
\N{w_n}_{\Ltwo[\Omega^n_\eta]}^2
\le(1+C\tau)\N{w_{n-1}}_{\Ltwo[\Omega^{n-1}_\eta]}^2
+C\tau^5\big(\NLtwo[D]{v_1}^2+\cdots +\NLtwo[D]{v_4}^2\big).
\end{equation*}
\end{lemma}
\begin{proof}
Note that $w_n=w_{n-1}\circ \BX^{n,n-1}_\tau$. The lemma can be proven by \eqref{eq:v-X0} and arguments similar to the proof of Lemma~\ref{lem:Uh}.
\end{proof}

\begin{lemma}\label{lem:v-vh}
Let Assumptions~\ref{ass-1} and \ref{ass-2} be satisfied. Suppose $h=O(\tau)=O(\eta^{4/5})$.
For any $v_h\in V_h$ and $v\in H^1(D)$,
\begin{align*}
\NLtwo[\Omega^n_\eta\backslash \Omega_{t_n} ]{v_h}^2
    \lesssim \tau^4\NLtwo[\Omega^n_h]{v_h}^2,\qquad
\NLtwo[\Omega^n_\eta\backslash \Omega_{t_n}]{v}^2
    \lesssim \tau^5\|v\|_{H^1(\Omega^n_\eta)}^2.
\end{align*}
\end{lemma}
\begin{proof}
The first inequality follows from Lemma~\ref{lem:bdr} and arguments similar to \eqref{eq:areErrvh}.
From \eqref{eq:Xerr} and \eqref{err-chi}, we know that $\text{dist}(\Gamma_\eta^n,\Gamma_{t_n})=O(\tau^5)$.
The second inequality is a direct consequence of \cite[Lemma~10 and (17)]{nic06}.
\end{proof}
\end{appendix}

\end{document}